\documentclass[10pt]{article}

\usepackage[margin=1.5in]{geometry}

\usepackage{amsmath}
\usepackage{amsthm}
\usepackage{amsfonts}
\usepackage{multicol}
\usepackage[numbers,sort&compress]{natbib}
\usepackage{booktabs}
\usepackage{titlesec}

\usepackage{enumitem}
\setlist[enumerate]{label={\upshape(\alph*)}}

\usepackage{tikz}
\usetikzlibrary{decorations.pathreplacing}

\usepackage{listings}
\usepackage{color}

\newcommand{\RT}{\mathrm{RT}}
\newcommand{\CRT}{\mathrm{CRT}}
\newcommand{\CRTS}{\mathrm{CRT}_{\mathrm{S}}}
\newcommand{\CRTI}{\mathrm{CRT}_{\mathrm{I}}}
\newcommand{\CRTW}{\mathrm{CRT}_{\mathrm{W}}}
\renewcommand{\circ}[1]{(#1)}
\newcommand{\f}{f_9}
\newcommand{\g}{f_{11}}
\newcommand{\F}{F}
\renewcommand{\d}{d_{19}}
\newcommand{\dt}{d_{23}}
\newcommand{\D}{D}
\renewcommand{\tt}[1]{\ensuremath{\mathtt{#1}}}
\newcommand{\p}{\mathbf{p}}
\newcommand{\pp}{\mathbf{P}}
\newcommand{\q}{\mathbf{q}}
\newcommand{\qq}{\mathbf{Q}}

\newcommand{\light}[1]{\underline{#1}}

\newtheorem{theorem}{Theorem}[section]
\newtheorem{lemma}[theorem]{Lemma}
\newtheorem{result}[theorem]{Result}

\theoremstyle{definition}
\newtheorem{definition}[theorem]{Definition}

\begin{document}

\title{Circular repetition thresholds on some small alphabets: Last cases of Gorbunova's conjecture}
\author{James~D.~Currie\footnote{Supported by NSERC, grant number 2017-03901}, Lucas~Mol, and Narad~Rampersad\footnote{Supported by NSERC, grant number 418646-2012}\\
\small University of Winnipeg\\
\small \{j.currie, l.mol, n.rampersad\}@uwinnipeg.ca}
\date{May 29, 2018}

\maketitle

\begin{abstract}
A word is called $\beta$-free if it has no factors of exponent greater than or equal to $\beta$.  The \emph{repetition threshold} $\RT(k)$ is the infimum of the set of all $\beta$ such that there are arbitrarily long $k$-ary $\beta$-free words (or equivalently, there are $k$-ary $\beta$-free words of every sufficiently large length, or even every length).  These three equivalent definitions of the repetition threshold give rise to three natural definitions of a repetition threshold for \emph{circular words}.  The infimum of the set of all $\beta$ such that 
\begin{enumerate}

\item there are arbitrarily long $k$-ary $\beta$-free circular words is called the \emph{weak circular repetition threshold}, denoted $\CRTW(k)$;

\item there are $k$-ary $\beta$-free circular words of every sufficiently large length is called the \emph{intermediate circular repetition threshold}, denoted $\CRTI(k)$;

\item there are $k$-ary $\beta$-free circular words of every length is called the \emph{strong circular repetition threshold}, denoted $\CRTS(k)$.

\end{enumerate}
We prove that $\CRTS(4)=\tfrac{3}{2}$ and $\CRTS(5)=\tfrac{4}{3}$, confirming a conjecture of Gorbunova and providing the last unknown values of the strong circular repetition threshold.  We also prove that $\CRTI(3)=\CRTW(3)=\RT(3)=\tfrac{7}{4}$.

\noindent
{\bf MSC 2010:} 68R15

\noindent
{\bf Keywords:} Circular words; Repetition threshold; Circular repetition threshold
\end{abstract}

\section{Introduction}

A (linear) word is simply a sequence of letters over some finite alphabet.  A word $w=w_1\dots w_n$, where the $w_i$ are letters, is \emph{periodic} if for some positive integer $p$, $w_{i+p}=w_i$ for all $1\leq i\leq n-p$.  In this case, $p$ is called a \emph{period} of $w$.  Note that every word of length $n$ trivially has period $n$.  The \emph{exponent} of a word $w$, denoted $\exp(w),$ is the ratio between its length and its minimal period.  If $r=\exp(w)$ is strictly greater than $1$, then $w$ is called an \emph{$r$-power}.  For example, the English word \tt{alfalfa} has minimal period $3$ and exponent $\tfrac{7}{3},$ so it is a $\tfrac{7}{3}$-power.

A word is called $\beta$-free if it has no factors of exponent greater than or equal to $\beta$, i.e.\ if it has no $r$-powers for $r\geq \beta$.  It is called $\beta^+$-free if it has no factors of exponent strictly greater than $\beta$.  The \emph{repetition threshold} function is given by
\[
\RT(k)=\inf\{\beta\colon \mbox{there are arbitrarily long $\beta$-free words on $k$ letters}\}.
\]
Equivalently, $\RT(k)$ is the smallest $\beta$ such that there is an infinite $\beta^+$-free word on $k$ letters.  It is well-known that the Thue-Morse sequence avoids overlaps~\cite{Berstel1995}, from which it follows that $\RT(2)=2.$  The repetition threshold function was introduced by Dejean~\cite{Dejean1972}, who conjectured that 
\[
\RT(k)=\begin{cases}
7/4 \mbox{ if } k=3;\\
7/5 \mbox{ if } k=4; \mbox{ and}\\
k/(k-1) \mbox{ if } k\geq 5.
\end{cases}
\]
Certain cases of the conjecture were proven by various authors~\cite{Dejean1972, Carpi2007, CurrieRampersad2009, CurrieRampersad2009again, Noori2007, Moulin1992, Pansiot1984}, and the last remaining cases were finally proven independently by Currie and Rampersad~\cite{CurrieRampersad2011} and Rao~\cite{Rao2011}.

Interest in Dejean's conjecture has given rise to a number of similar threshold problems in combinatorics on words.  These include the \emph{generalized repetition threshold} introduced by Ilie, Ochem, and Shallit~\cite{IlieOchemShallit2005}, and the \emph{Abelian repetition threshold} introduced by Samsonov and Shur~\cite{SamsonovShur2012}.  In this article, we are concerned with a repetition threshold for \emph{circular words}.  Two words $x$ and $y$ are said to be \emph{conjugates} if there are words $u$ and $v$ such that $x=uv$ and $y=vu$.  The conjugates of a word $w$ can be obtained by rotating the letters of $w$ cyclically. For a word $w$, the \emph{circular word} $(w)$ is the set of all conjugates of $w$.  Intuitively, one can think of a circular word as being obtained from a linear word by linking the ends, giving a cyclic sequence of letters.

A word is a \emph{factor} of a circular word $\circ{w}$ if it is a factor of some conjugate of $w$.  As for linear words, a circular word is $\beta$-free if it has no factors of exponent greater than or equal to $\beta$, and $\beta^+$-free if it has no factors of exponent strictly greater than $\beta$.  

Note that if the linear word $w$ is $\beta$-free, then so are all of its factors. This means that there are three equivalent definitions of the repetition threshold $\RT(k)$.  It is the infimum of the set of all $\beta$ such that
\begin{enumerate}
\item there are arbitrarily long $k$-ary $\beta$-free words;
\item there are $k$-ary $\beta$-free words of every sufficiently large length; or
\item there are $k$-ary $\beta$-free words of every length.
\end{enumerate} 
On the other hand, if $x$ is a factor of some $\beta$-free circular word $(w)$, it is not necessarily true that the circular word $(x)$ is $\beta$-free (even though the linear word $x$ must be $\beta$-free).  Thus, there are three natural definitions of a repetition threshold for circular words:
\begin{enumerate}
\item the \emph{weak circular repetition threshold}, 
\[
\CRTW(k)=\inf\{\beta\colon\ \mbox{there are arbitrarily long $k$-ary $\beta$-free circular words}\};
\]

\item the \emph{intermediate circular repetition threshold},
\[
\CRTI(k)=\inf\{\beta\colon\ \mbox{there are $k$-ary $\beta$-free circular words of every sufficiently large length}\};
\] 

\item and the \emph{strong circular repetition threshold},
\[
\CRTS(k)=\inf\{\beta\colon\ \mbox{there are $k$-ary $\beta$-free circular words of every length}\}.
\] 
\end{enumerate}
Clearly we have
\begin{align}
\RT(k)\leq \CRTW(k)\leq \CRTI(k)\leq \CRTS(k) \label{CRTorder}
\end{align}
for all $k\geq 2.$

Almost all values of the strong circular repetition threshold $\CRTS(k)$ are known. Aberkane and Currie~\cite{AberkaneCurrie2004} demonstrated that $\CRTS(2)=5/2$, while the fact that $\CRTS(3)=2$ follows from the work of Currie~\cite{Currie2002} along with a finite search, or alternatively from the work of Shur~\cite{Shur2011}.  Gorbunova~\cite{Gorbunova2012} demonstrated that $\CRTS(k)=\frac{\lceil k/2\rceil+1}{\lceil k/2\rceil}$ for all $k\geq 6$, and conjectured that this formula holds for $k=4$ and $k=5$ as well.  In this article, we demonstrate that $\CRTS(4)=\tfrac{3}{2}$ and $\CRTS(5)=\tfrac{4}{3}$, confirming Gorbunova's conjecture.  Since $\CRT(4)\geq \tfrac{3}{2}$ and $\CRTS(5)\geq \tfrac{4}{3}$ are already known~\cite{Gorbunova2012}, we only need to show that there are $\tfrac{3}{2}^+$-free circular $4$-ary words of every length, and $\tfrac{4}{3}^+$-free circular $5$-ary words of every length.

Relatively less is known about the intermediate and weak circular repetition thresholds, even though the fact that $\CRTW(2)=2$ was noted in the work of Thue~(see \cite{Berstel1995}).  Aberkane and Currie~\cite{AberkaneCurrie2005} demonstrated that $\CRTI(2)=7/3.$  In \cite[Section 4]{Shur2011}, it is stated that $\CRTI(3)=\tfrac{7}{4}$, and that the result is obtained by adapting the technique of~\cite{Shur2010}, though the proof is omitted due to space constraints.  We give an alternate proof that $\CRTI(3)=\CRTW(3)=\RT(3)=\tfrac{7}{4}$ by adapting the method we use to prove $\CRTS(4)=\tfrac{3}{2}$.

All that is known about $\CRTI(k)$ and $\CRTW(k)$ for $k\geq 4$ are the bounds given by~(\ref{CRTorder}); note that $\RT(k)=\frac{k}{k-1}$ for all $k\geq 5$ and $\CRTS(k)=\frac{\lceil k/2\rceil+1}{\lceil k/2\rceil}$ for all $k\geq 6$, so
\[
\frac{k}{k-1}\leq \CRTW(k)\leq \CRTI(k)\leq \frac{\lceil k/2\rceil+1}{\lceil k/2\rceil} \mbox{ for all } k\geq 6.
\]
With the knowledge that $\CRTI(3)=\CRTW(3)=\RT(3)$, it seems reasonable to conjecture that $\CRTI(k)=\CRTW(k)=\RT(k)$ for all $k\geq 4$.  We note that this strengthens statement (2) of Conjecture 1 in~\cite{Shur2011}. However, it is likely that different techniques than those used here will be needed to prove this conjecture.


\section{$\CRTS(4)=\tfrac{3}{2}$}\label{4}

We first give a short description of the technique used to achieve the main result of this section.  We use a strong inductive argument to demonstrate that there are circular $\tfrac{3}{2}^+$-free $4$-ary words of every length.  The inductive step involves constructing longer circular $\tfrac{3}{2}^+$-free words from shorter ones.  We use uniform morphisms to do so.  However, we need two uniform morphisms of distinct (and relatively prime) sizes in order to show that there is a circular $\tfrac{3}{2}^+$-free $4$-ary word of every length.  First, we find a $9$-uniform morphism $\f$ and an $11$-uniform morphism $\g$ that preserve $\tfrac{3}{2}^+$-freeness.  To construct a $\tfrac{3}{2}^+$-free word of length $n$, we write $n=9k+11\ell$ for integers $k$ and $\ell$.  This is possible by the following well-known lemma; we use the version stated in~\cite{Skupien1993}.

\begin{lemma}\label{postage}
If $a,b\in\mathbb{N}$ and $\gcd(a,b) = 1$, then for each
$n\geq (a-1)(b-1)$, there is exactly one pair of nonnegative integers $r$ and $s$ such that $s<a$ and $n=ra+bs.$ \hfill \qed
\end{lemma}

\noindent
Finally, we take a $\tfrac{3}{2}^+$-free circular word $(w)$ of length $k+\ell$, write $w=uv$ where $|u|=k$ and $|v|=\ell$, and show that $(\f(u)\g(v))$ is also $\tfrac{3}{2}^+$-free.  We introduce some terminology for dealing with these ``mixed'' images of $\f$ and $\g$.

Let $A$ and $B$ be alphabets, and let $h\colon A^*\rightarrow B^*$ be a morphism.  Using the standard notation for images of sets, we have $h(A)=\{h(a)\colon\ a\in A\},$ which we refer to as the set of \emph{building blocks} of $h$.  If $H$ is a \emph{set} of morphisms from $A^*$ to $B^*$, then we let $H(A)$ denote the set of all images of letters of $A$ under all morphisms in $H$; that is,
\[
H(A)=\bigcup_{h\in H}h(A)=\{h(a)\colon\ h\in H, a\in A\}.
\]
We call the words in $H(A)$ the \emph{building blocks} of $H$.

\begin{definition}
Let $A$ and $B$ be alphabets and let $H$ be a set of nonerasing morphisms from $A^*$ to $B^*$.  An \emph{$H$-image} of a word $w=a_1\dots a_n$ with $a_i\in A$ is a word of the form
\[
h_1(a_1)\dots h_n(a_n),
\]
where $h_i\in H$ (note that the $h_i$ are not necessarily distinct).  Let $u$ be a nonempty factor of some word $v\in H(A)^+.$  We say that $w=a_1\dots a_n$ is an \emph{$H$-preimage} of $u$ if there is an $H$-image $h_1(a_1)\dots h_n(a_n)$ of $w$ that has $u$ as a factor, and $w$ is minimal in the sense that $h_1(a_1)\dots h_{n-1}(a_{n-1})$ and $h_2(a_2)\dots h_n(a_n)$ do not have $u$ as a factor.  The $H$-preimage of the empty word $\tt{\epsilon}$ is simply $\tt{\epsilon}$.
\end{definition}

\bigskip

We are now ready to define $\f$ and $\g$ and start working towards the proof that $\CRTS(4)=\tfrac{3}{2}$.  Let $\Sigma_4=\{\tt{0},\tt{1},\tt{2},\tt{3}\}$.  We define the morphism $\f:\Sigma_4^*\rightarrow \Sigma_4^*$ by
\begin{align*}
\tt{0}&\mapsto \tt{0121323\ 10}\\
\tt{1}&\mapsto \tt{1232030\ 21}\\
\tt{2}&\mapsto \tt{2303101\ 32}\\
\tt{3}&\mapsto \tt{3010212\ 03}
\end{align*}
and the morphism $\g:\Sigma_4^*\rightarrow \Sigma_4^*$ by
\begin{align*}
\tt{0}&\mapsto \tt{0121323\ 12\ 10}\\
\tt{1}&\mapsto \tt{1232030\ 23\ 21}\\
\tt{2}&\mapsto \tt{2303101\ 30\ 32}\\
\tt{3}&\mapsto \tt{3010212\ 01\ 03}
\end{align*}
Throughout this section, we let $\F=\{\f,\g\}.$  We make the following observations:
\begin{itemize}
\item The morphisms $\f$ and $\g$ both have cyclic structure: for $a\in\{\tt{1},\tt{2},\tt{3}\},$ $\f(a)$ is obtained from $\f(\tt{0})$ by adding $a$ to each letter of $\f(\tt{0})$ modulo $4$, and likewise for $\g$.
\item The word $\g(\tt{0})$ is obtained from $\f(\tt{0})$ by inserting the factor $\tt{12}$ after the length $7$ prefix (indicated above with spacing).
\item For all $a\in\Sigma_4$, $\f(a)$ and $\g(a)$ have the same prefix of length $8$, and the same suffix of length $2$.
\item $F(\Sigma_4)$ is a prefix code; the length $9$ prefixes of all building blocks of $F$ are distinct.
\item $F(\Sigma_4)$ is a suffix code; the length $3$ suffixes of all building blocks of $F$ are distinct.
\end{itemize}


If $P$ is a prefix code and $w$ is a nonempty factor of some element of $P^+$, a \emph{cut} of $w$ is a pair $(x,y)$ such that (i) $w=xy$; and (ii) for any words $p,s$ with $pws\in P^+,$ $px\in P^*$.  We use vertical bars to denote cuts.  For example, over the code $\{\tt{01},\tt{10}\},$ the word $\tt{11}$ has cut $\tt{1}\vert \tt{1}$.

We first prove that any sufficiently long factor of a word in $(F(\Sigma_4))^+$ has a cut, and that any factor of a word in $(F(\Sigma_4))^+$ with a cut has unique $F$-preimage.  These results are used frequently in the material that follows, sometimes without reference.

\begin{lemma}\label{cut}
Let $u$ be a factor of some $\F$-image.  If $|u|\geq 10$, then $u$ has a cut.
\end{lemma}

\begin{proof}
Consider the set $P=\{\tt{0121},\tt{1232},\tt{2303},\tt{3010}\}$ of all length $4$ prefixes of the building blocks of $\F$.  Observe that there is a cut to the left of every appearance of a member of $P$ in $u$, since each member of $P$ appears as a factor in an $\F$-image only as the prefix of a building block (it suffices to check all $\F$-images of words of length $2$).  On the other hand, consider the set
\[
S=\{\tt{2310},\tt{3021},\tt{0132},\tt{1203},\tt{1210},\tt{2321},\tt{3032},\tt{0103}\}
\]
of all length $4$ suffixes of the building blocks of $\F$.  There is a cut to the right of every appearance of a member of $S$ in $u$, since each member of $S$ appears as a factor in an $\F$-image only as the suffix of a building block (again, we need only check all $\F$-images of words of length $2$).

Let $|u|\geq 10.$  Since the building blocks of $\F$ have length at most $11$, $u$ must contain either the prefix of length $4$ of some building block of $\F$ (in which case there is a cut to the left of this factor), or the suffix of length $4$ of some building block of $\F$ (in which case there is a cut to the right of this factor).
\end{proof}

\begin{lemma}\label{preimage}
Let $u$ be a factor of some $\F$-image.  If $u$ has a cut, then $u$ has unique $F$-preimage.
\end{lemma}

\begin{proof}
Suppose that $u$ has a cut.  Since $\F(\Sigma_4)$ is a bifix code, we can write $u=s\vert v\vert p,$ where $s$ is a proper suffix of some word in $\F(\Sigma_4)$, $v\in \F(\Sigma_4)^*$, and $p$ is a proper prefix of some word in $\F(\Sigma_4)$.  Since $v\in \F(\Sigma_4)^*$ and $\F(\Sigma_4)$ is a code, $v$ has unique $\F$-preimage.  Note that for any letter $a\in\Sigma_4$, both $\f(a)$ and $\g(a)$ begin and end in $a$.  Thus if $s$ is nonempty, then the $\F$-preimage of $s$ is completely determined by the last letter of $s$, while if $p$ is nonempty, then the $\F$-preimage of $p$ is completely determined by the first letter of $p$.
\end{proof}

Next, we show that the individual morphisms $\f$ and $\g$ are $\tfrac{3}{2}^+$-free.  That is, they preserve $\tfrac{3}{2}^+$-freeness for linear words.

\begin{lemma}\label{Ordinary}
Let $w\in \Sigma_4^+$.  If $w$ is $\tfrac{3}{2}^+$-free, then $\f(w)$ and $\g(w)$ are $\tfrac{3}{2}^+$-free.
\end{lemma}

\begin{proof}
Let $w$ be $\tfrac{3}{2}^+$-free, and let $h\in\{\f,\g\}$.  Suppose towards a contradiction that $h(w)$ has a factor of exponent greater than $\tfrac{3}{2}$.  Then in particular, $h(w)$ has a factor of the form $xyx$ for words $x,y\in\Sigma_4^*$ with $|xyx|>\tfrac{3}{2}|xy|$, or equivalently, $|x|>|y|.$

First suppose $|x|\leq 9$.  Then $|y|\leq 8$ and $|xyx|\leq 26$, so $xyx$ appears in some word $h(u)$ with $u\in\Sigma_4^+$ a $\tfrac{3}{2}^+$-free word of length $4$.  Eliminating this possibility by an exhaustive search, we may assume that $|x|\geq 10$.  Then by Lemma~\ref{cut}, $x$ has a cut.  So $x$ has the form $s_x\vert m_x \vert p_x$, where $s_x$ is a proper suffix of some building block of $h$ and $p_x$ is a proper prefix of some building block of $h$.   This means that $y=s_y\vert m_y\vert p_y$, where $p_xs_y$ and $p_ys_x$ are either empty, or building blocks of $h$.  Let $m_x$ and $m_y$ have $h$-preimages $x'$ and $y'$, respectively, and let $p_xs_y$ and $p_ys_x$ have $h$-preimages $a$ and $b$, respectively.  Note that $a,b\in\Sigma_4\cup\{\tt{\epsilon}\}.$  Then $xyx$ has $h$-preimage
\[
bx'ay'bx'a.
\]
Since $h$ is uniform and $|x|>|y|,$ it follows that $|bx'a|>|y'|$, and this contradicts the assumption that $w$ is $\tfrac{3}{2}^+$-free.
\end{proof}

Now that we know that $\f$ and $\g$ preserve $\tfrac{3}{2}^+$-freeness for linear words, we are ready to show that we can construct longer $\tfrac{3}{2}^+$-free circular words from shorter ones using $\f$ and $\g$.  For a word $w$ of length $n$ and an integer $m\in\{1,\dots,n\},$ we let $p_m(w)$ denote the prefix of $w$ of length $m$, and we let $s_m(w)$ denote the suffix of $w$ of length $m$.

\begin{theorem}\label{preserve}
Let $k\geq 8$ and $2\leq \ell\leq 10$.  Let $u,v\in\Sigma_4^+$ be words of length $k$ and $\ell$, respectively.  If the circular word $\circ{uv}$ is $\tfrac{3}{2}^+$-free, then so is the circular word $\circ{\f(u)\g(v)}.$
\end{theorem}

\begin{proof}
Let $u=u_1u_2\dots u_k$ and $v=v_1v_2\dots v_\ell$ and suppose that $\circ{uv}$ is $\tfrac{3}{2}^+$-free.  For ease of notation, let $U=\f(u)$ and $V=\g(v).$ 
Suppose towards a contradiction that $\circ{UV}$ has a factor of exponent greater than $\tfrac{3}{2}.$  Then some conjugate of $UV$ has a factor $xyx$ with $|x|>|y|$.  Suppose first that $|x|\leq 10$.  Then $|y|\leq 9$ and thus $|xyx|\leq 29.$  It follows that $xyx$ is a factor of some $F$-image $\f(w_1)\g(w_2)\f(w_3)$, where $w_1,w_2,w_3 \in \Sigma_4^*$ satisfy $|w_1|+|w_2|+|w_3|=5$ and $w_1w_2w_3$ is $\tfrac{3}{2}^+$-free.  By exhaustive search, we eliminate this possibility and we may now assume that $|x|\geq 11.$

We first claim that $x$ cannot contain the factor $s_3(U)p_9(V).$  By inspection, $s_3(U)p_9(V)$ has the cut $s_3(U)\vert p_9(V)$.  No matter the exact identity of $u_k$, the suffix $s_3(U)$ is not a suffix of any building block of $\g$.  Similarly, the prefix $p_9(V)$ is not a prefix of any building block of $\f$, no matter the identity of $v_1$.  So the only place that the factor $s_3(U)p_9(V)$ appears as a factor of the circular word $(UV)$ is at the boundary between $U$ and $V$. Since $x$ appears at least twice in some conjugate of $UV$ (as $xyx$ is a factor of some conjugate of $UV$), we conclude that $s_3(U)p_9(V)$ is not a factor of $x$.  By a similar argument, $x$ cannot contain the factor $s_3(V)p_9(U).$

\begin{center}
\begin{tikzpicture}[scale=0.85]
\draw (0,0.6) rectangle (7,1.2) node[pos=.5] {$U$};
\draw (7,0.6) rectangle (15.5,1.2) node[pos=.5] {$V$};
\draw (0,0) rectangle (2,0.6) node[pos=.5] {$\f(u_1)$};
\draw (2,0) rectangle (4,0.6) node[pos=.5] {$\f(u_2)$};
\draw (4.5,0.3) node {$\cdots$};
\draw (5,0) rectangle (7,0.6) node[pos=.5] {$\f(u_k)$};
\draw (7,0) rectangle (9.5,0.6) node[pos=.5] {$\g(v_1)$};
\draw (9.5,0) rectangle (12,0.6) node[pos=.5] {$\g(v_2)$};
\draw (12.5,0.3) node {$\cdots$};
\draw (13,0) rectangle (15.5,0.6) node[pos=.5] {$\g(v_\ell)$};
\draw (0,0) rectangle (2,-0.6) node[pos=.5] {\footnotesize $p_9(U)$};
\draw (6,0) rectangle (7,-0.6) node[pos=.5] {\footnotesize $s_3(U)$};
\draw (7,0) rectangle (9,-0.6) node[pos=.5] {\footnotesize $p_9(V)$};
\draw (14.5,0) rectangle (15.5,-0.6) node[pos=.5] {\footnotesize $s_3(V)$};
\end{tikzpicture}
\end{center}

Since $|x|\geq 11,$ $x$ must contain either the prefix of length $9$ of some building block of $F$, or the suffix of length $3$ of some building block of $F.$  If $x$ contains the length $9$ prefix or the length $3$ suffix of some building block of $\f,$ then $x$ appears only inside $s_2(V)Up_8(V)$ by the argument of the previous paragraph.  On the other hand, if $x$ contains the length $9$ prefix or the length $3$ suffix of some building block of $\g$, then $x$ appears only inside $s_2(U)Vp_8(U).$  This analysis leads to four separate cases, depending on the positioning of the entire factor $xyx$.

\noindent
\textbf{Case 1:} $xyx$ is a factor of $s_2(V) U p_8(V)$.

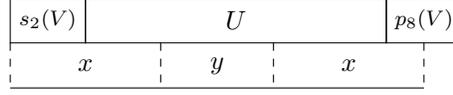
\begin{center}
\begin{tikzpicture}
\draw (0,0) rectangle (1,0.6) node[pos=.5] {\footnotesize $s_2(V)$};
\draw (1,0) rectangle (5,0.6) node[pos=.5] {$U$};
\draw (5,0) rectangle (6,0.6) node[pos=.5] {\footnotesize $p_8(V)$};

\draw[dashed] (0,0) -- (0,-0.6);
\draw[dashed] (2,0) -- (2,-0.6);
\draw[dashed] (3.5,0) -- (3.5,-0.6);
\draw[dashed] (5.5,0) -- (5.5,-0.6);

\draw (1,-0.3) node {$x$};
\draw (2.75,-0.3) node {$y$};
\draw (4.5,-0.3) node {$x$};
\draw (0,-0.6) -- (5.5,-0.6);
\end{tikzpicture}
\end{center} 

\noindent
Observe that $s_2(V)$ is the suffix of length $2$ of $\g(v_\ell)$, and that this is the same as the length $2$ suffix of $\f(v_\ell)$.  Similarly, $p_8(V)$ is the prefix of length $8$ of both $\g(v_1)$ and $\f(v_1)$.  Thus $xyx$ is a factor of $\f(v_\ell)U\f(v_1)=\f(v_\ell)\f(u)\f(v_1)=\f(v_\ell uv_1)$.  By Lemma~\ref{Ordinary}, $v_\ell uv_1$ must have a factor with exponent greater than $\tfrac{3}{2}$.  But since $\ell\geq 2$, $v_\ell uv_1$ is a factor of $(uv)$, and this contradicts the assumption that $(uv)$ is $\tfrac{3}{2}^+$-free.

\medskip

\noindent
\textbf{Case 2:} $x$ is a factor of $s_2(V)Up_8(V)$ and $xyx$ has $V$ as a factor.

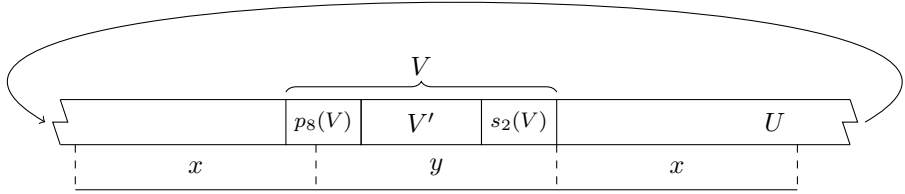
\begin{center}
\begin{tikzpicture}
\path[use as bounding box] (-4,2) rectangle (8.5,-0.6);
\draw (-3,0) -- (0,0);
\draw (-3,0.6) -- (0,0.6);
\draw (-3,0) -- (-3.1,0.3) -- (-2.9,0.3) -- (-3,0.6);
\draw (1,0) rectangle (2.6,0.6) node[pos=.5] {$V'$};
\draw (3.6,0) -- (7.5,0);
\draw (3.6,0.6) -- (7.5,0.6);
\draw (7.5,0) -- (7.4,0.3) -- (7.6,0.3) -- (7.5,0.6);
\draw (0,0) rectangle (1,0.6) node[pos=.5] {\footnotesize $p_8(V)$};
\draw (2.6,0) rectangle (3.6,0.6) node[pos=.5] {\footnotesize $s_2(V)$};

\draw[dashed] (-2.8,0) -- (-2.8,-0.6);
\draw[dashed] (0.4,0) -- (0.4,-0.6);
\draw[dashed] (3.6,0) -- (3.6,-0.6);
\draw[dashed] (6.8,0) -- (6.8,-0.6);
\draw (-1.2,-0.3) node {$x$};
\draw (2,-0.3) node {$y$};
\draw (5.2,-0.3) node {$x$};
\draw (-2.8,-0.6) -- (6.8,-0.6);

\draw[->] (7.7,0.3) to[out=30,in=150] (-3.2,0.3);
\draw (6.5,0.3) node {$U$};

\draw [decorate,decoration={brace,amplitude=4pt}]
(0,0.7) -- (3.6,0.7) node [above,black,midway,yshift=3pt]{$V$};
\end{tikzpicture}
\end{center}

\noindent
Let $x=s_x\vert m_x\vert p_x,$ and $y=s_y\vert m_y\vert p_y$, where $s_x$ and $s_y$ are proper suffixes of building blocks of $F$ and $p_x$ and $p_y$ are proper prefixes of building blocks of $F$.  Then 
\[
xyx=s_x\vert m_x \vert p_xs_y \vert m_y \vert p_ys_x \vert m_x \vert p_x
\]
Let $m_x$ have $F$-preimage $x'$ and $m_y$ have $F$-preimage $y'$.  Then $|m_x|=9|x'|$ and $|m_y|\geq 9|y'|.$  Now $p_xs_y$ and $p_ys_x$ are each either a single building block or the empty word.  Let $p_xs_y$ have $F$-preimage $a\in\Sigma_4\cup\{\tt{\epsilon}\}$ and $p_ys_x$ have $F$-preimage $b\in\Sigma_4\cup \{\tt{\epsilon}\}.$  Then the $F$-preimage of $xyx$ is
\[
bx'ay'bx'a.
\]
Note that this preimage may not be a factor of the circular word $(uv)$ if $xyx$ is so long that the factor $p_x$ at the end of $xyx$ and the factor $s_x$ at the beginning of $xyx$ are actually part of the same building block $\f(u_i).$  We break into two subcases:

\noindent
\textbf{Subcase 2a:} $bx'ay'bx'a$ is a factor of $(uv).$

We argue that $|bx'a|>|y'|,$ which contradicts the assumption that $(uv)$ is $\tfrac{3}{2}^+$-free.  We have
\[
9|bx'a|\geq |s_x|+|m_x|+|p_x|= |x| >|y|\geq |m_y| \geq 9|y'|,
\]
from which the desired inequality follows.

\noindent
\textbf{Subcase 2b:} $bx'ay'bx'a$ is not a factor of $(uv).$

Then neither $a$ nor $b$ is equal to $\tt{\epsilon}$, since the ends of $xyx$ are in the same building block $\f(u_i)$.  Note that both $x'ay'bx'a$ and $bx'ay'bx'$ appear in $(uv).$  We claim that $|x'a|>|y'b|,$ which contradicts the assumption that $(uv)$ is $\tfrac{3}{2}^+$-free.  Since $|x|>|y|,$ we have $|s_xm_xp_x|>|s_ym_yp_y|.$  Thus, $|m_x|>|m_y|+|s_y|+|p_y|-|s_x|-|p_x|$.  Further, since $p_x$ and $s_x$ appear in the same building block $\f(u_i)$ and must not overlap, we have $|p_x|+|s_x|\leq 9.$  Since $|p_xs_y|\geq 9$ and $|p_ys_x|\geq 9$ we have $|s_y|\geq 9-|p_x|$ and $|p_y|\geq 9-|s_x|,$ so
\[
|m_x|>|m_y|+|s_y|+|p_y|-|s_x|-|p_x|\geq |m_y|+18-2(|p_x|+|s_x|)\geq |m_y|.
\]
Now
\[
9|x'|= |m_x|>|m_y|\geq 9|y'|,
\]
and $|x'a|=|x'|+1>|y'|+1=|y'b|$ follows.
\medskip

\noindent
\textbf{Case 3:} $xyx$ is a factor of $s_2(U) Vp_8(U)$.

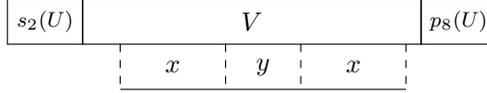
\begin{center}
\begin{tikzpicture}
\draw (-0.5,0) rectangle (0.5,0.6) node[pos=.5] {\footnotesize $s_2(U)$};
\draw (0.5,0) rectangle (5,0.6) node[pos=.5] {$V$};
\draw (5,0) rectangle (6,0.6) node[pos=.5] {\footnotesize $p_8(U)$};

\draw[dashed] (1,0) -- (1,-0.6);
\draw[dashed] (2.4,0) -- (2.4,-0.6);
\draw[dashed] (3.4,0) -- (3.4,-0.6);
\draw[dashed] (4.8,0) -- (4.8,-0.6);

\draw (1.7,-0.3) node {$x$};
\draw (2.9,-0.3) node {$y$};
\draw (4.1,-0.3) node {$x$};
\draw (1,-0.6) -- (4.8,-0.6);
\end{tikzpicture}
\end{center}

\noindent
Observe that $s_2(U)$ is the suffix of length $2$ of $\f(u_k)$, and that this is the same as the length $2$ suffix of $\g(u_k)$.  Similarly, $p_8(U)$ is the prefix of length $8$ of both $\f(u_1)$ and $\g(u_1)$.  Thus, $xyx$ is a factor of $\g(u_k)V\g(u_1)=\g(u_k)\g(v)\g(u_1)=\g(u_k v u_1)$. By Lemma~\ref{Ordinary}, $u_kvu_1$ must have a factor with exponent greater than $\tfrac{3}{2}$.  But since $k\geq 8$, $u_k vu_1$ is a factor of $(uv)$, and this contradicts the assumption that $(uv)$ is $\tfrac{3}{2}^+$-free.

\medskip

\noindent
\textbf{Case 4:}  $x$ is a factor of $s_2(U)V p_8(U)$ and $xyx$ has $U$ as a factor.

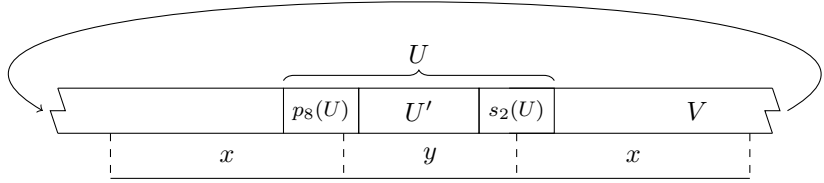
\begin{center}
\begin{tikzpicture}
\path[use as bounding box] (-4,2) rectangle (7.5,-0.6);
\draw (0,0) -- (-3,0);
\draw (-3,0) -- (-3.1,0.3) -- (-2.9,0.3) -- (-3,0.6);
\draw (-3,0.6) -- (0,0.6);
\draw (3,0) -- (6.5,0);
\draw (6.5,0) -- (6.4,0.3) -- (6.6,0.3) -- (6.5,0.6);
\draw (6.5,0.6) -- (3,0.6);
\draw (5.5,0.3) node {$V$};
\draw (1,0) rectangle (2.6,0.6) node[pos=.5] {$U'$};

\draw (0,0) rectangle (1,0.6) node[pos=.5] {\footnotesize $p_8(U)$};
\draw (2.6,0) rectangle (3.6,0.6) node[pos=.5] {\footnotesize $s_2(U)$};

\draw[dashed] (-2.3,0) -- (-2.3,-0.6);
\draw[dashed] (0.8,0) -- (0.8,-0.6);
\draw[dashed] (3.1,0) -- (3.1,-0.6);
\draw[dashed] (6.2,0) -- (6.2,-0.6);
\draw (-0.75,-0.3) node {$x$};
\draw (1.95,-0.3) node {$y$};
\draw (4.65,-0.3) node {$x$};
\draw (-2.3,-0.6) -- (6.2,-0.6);
\draw[->] (6.7,0.3) to[out=30,in=150] (-3.2,0.3);

\draw [decorate,decoration={brace,amplitude=4pt}]
(0,0.7) -- (3.6,0.7) node [above,black,midway,yshift=3pt]{$U$};
\end{tikzpicture}
\end{center}

\noindent
Since $y$ contains $U'$ as a factor,
\[
|y|\geq |U'|=|U|-10=|\f(u)|-10=9|u|-10=9k-10\geq 62,
\]
from the assumption that $|u|=k\geq 8.$  On the other hand, since $x$ appears twice (without overlapping itself) in $s_2(U)Vp_8(U),$ we conclude that
\[
|x|\leq \frac{|V|+10}{2}=\frac{|\g(v)|+10}{2}=\frac{11|v|+10}{2}=\frac{11\ell+10}{2}\leq 60,
\]
from the assumption that $|v|=\ell\leq 10.$  But this contradicts the assumption that $|x|>|y|.$
\end{proof}

%
%

We are now ready to prove the main result of this section.

\begin{theorem}
For every $n\in\mathbb{N}$, there is a $\tfrac{3}{2}^+$-free circular $4$-ary word of length $n$.
\end{theorem}

\begin{proof}
The proof is by strong induction on $n$.  For $n\leq 173$ we found a $\tfrac{3}{2}^+$-free circular $4$-ary word of length $n$ by computer search.

Assume that for some $n\geq 174$, there is a $\tfrac{3}{2}^+$-free circular $4$-ary word of every length $m<n.$  Then $n-9(8)-11(2)\geq 174-72-22=80,$ so by Lemma~\ref{postage}, we can write $n-9(8)-11(2)=9r+11s,$ or equivalently $n=9(r+8)+11(s+2)$, for integers $r\geq 0$ and $0\leq s \leq 8$.  Let $k=r+8$ and $\ell=s+2,$ and note that $k\geq 8$ and $2\leq \ell\leq 10.$  Clearly, $k+\ell<n$, so by the inductive hypothesis, there is a $\tfrac{3}{2}^+$-free circular $4$-ary word $(w)$ of length $k+\ell$.  Let $w=uv$, with $|u|=k$ and $|v|=\ell.$  By Theorem~\ref{preserve}, the circular $4$-ary word $(\f(u)\g(v))$ is also $\tfrac{3}{2}^+$-free, and has length $9k+11\ell=n$.  
\end{proof}

\section{$\CRTI(3)=\tfrac{7}{4}$}\label{3}

It follows immediately from the fact that $\RT(3)=\tfrac{7}{4}$ that $\CRTI(3)\geq\tfrac{7}{4}$.  Here, we demonstrate that there are $\tfrac{7}{4}^+$-free ternary words of every length $n\geq 23$, from which we conclude that $\CRTI(3)=\tfrac{7}{4}.$  The only lengths which do not admit a $\tfrac{7}{4}^+$-free circular ternary word are $5$, $7,$ $9,$ $10,$ $14$, $16,$ $17,$ and $22$.

We use a construction very similar to the one used in the previous section.  Let $\Sigma_3=\{\tt{0},\tt{1},\tt{2}\}.$  Define the morphism $\d:\Sigma_3^*\rightarrow \Sigma_3^*$ by
\begin{align*}
\tt{0}&\mapsto \tt{012021201\ 2102120210}\\
\tt{1}&\mapsto \tt{120102012\ 0210201021}\\
\tt{2}&\mapsto \tt{201210120\ 1021012102}
\end{align*}
and the morphism $\dt:\Sigma_3^*\rightarrow \Sigma_3^*$ by
\begin{align*}
\tt{0}&\mapsto \tt{012021201 \ 0201 \ 2102120210}\\
\tt{1}&\mapsto \tt{120102012 \ 1012 \ 0210201021}\\
\tt{2}&\mapsto \tt{201210120 \ 2120 \ 1021012102}
\end{align*}
Throughout this section, let $\D=\{\d,\dt\}.$  Note that $\d$ is the morphism used by Dejean~\cite{Dejean1972} to prove that $\RT(3)=\tfrac{7}{4}$.  In particular, Dejean proved that $\d$ is $\tfrac{7}{4}^+$-free.  The image $\dt(\tt{0})$ is obtained by inserting the factor $\tt{0201}$ into the middle of $\d(\tt{0})$ (indicated above by spacing), and for $a\in\{\tt{1},\tt{2}\}$, $\dt(a)$ is obtained from $\dt(\tt{0})$ by adding $a$ to each letter of $\dt(\tt{0})$ modulo $3$.  Note that $\d$ and $\dt$ have similar properties to $\f$ and $\g$ used in Section~\ref{4}:  
\begin{itemize}
\item Both $\d$ and $\dt$ have the cyclic structure described above for $\dt$.
\item For all $a\in\Sigma_3$, $\d(a)$ and $\dt(a)$ have the same prefix of length $9$, and the same suffix of length $13$.
\item $D(\Sigma_3)$ is a prefix code; the length $10$ prefixes of all building blocks of $D$ are distinct.
\item $D(\Sigma_3)$ is a suffix code; the length $14$ suffixes of all building blocks of $D$ are distinct.
\end{itemize}

\begin{lemma}
Let $u$ be a factor of some $D$-image.
\begin{enumerate}
\item If $|u|\geq 22$, then $u$ has a cut. \label{Dcut}
\item If $u$ has a cut, then $u$ has unique $D$-preimage.  \label{Dpre}
\end{enumerate}
\end{lemma}

\begin{proof}
The proof of~\ref{Dcut} is similar to that of Lemma~\ref{cut}.  If $|u|\geq 22$ then $u$ contains the length $9$ prefix of some building block (and there is a cut to the left of this prefix) or the length $9$ suffix of some building block (and there is a cut to the right of this prefix).  The proof of~\ref{Dpre} is similar to that of Lemma~\ref{preimage}.
\end{proof}

\begin{theorem}\label{preserve3}
Let $k\geq 6$ and $2\leq \ell\leq 20$.  Let $u,v\in\Sigma_3^+$ be words of length $k$ and $\ell$, respectively.  If the circular word $\circ{uv}$ is $\tfrac{7}{4}^+$-free, then so is the circular word $\circ{\d(u)\dt(v)}.$
\end{theorem}

\begin{proof}
Let $u=u_1\dots u_k$ and $v=v_1\dots v_\ell$ and suppose that $(uv)$ is $\tfrac{7}{4}^+$-free.  Let $U=\d(u)$ and $V=\dt(v).$  Suppose towards a contradiction that $(UV)$ has a factor of exponent greater than $\tfrac{7}{4}.$  Then some conjugate of $UV$ has a factor $xyx$ with $|x|>3|y|.$  Suppose first that $|x|\leq 22$.  Then $|y|\leq 7$ and $|xyx|\leq 51.$  It follows that $xyx$ is a factor of some $\D$-image of the form $\d(w_1)\dt(w_2)\d(w_3)$, where $w_1,w_2,w_3\in\Sigma_3^*$ satisfy $|w_1|+|w_2|+|w_3|=4$ and $w_1w_2w_3$ is $\tfrac{7}{4}^+$-free. We eliminate this possibility by exhaustive search and may now assume that $|x|\geq 23.$

Notice that $x$ contains neither $s_{14}(U)p_{10}(V)$ nor $s_{14}(V)p_{10}(U)$ as a factor, since each of these factors appears at most once in any conjugate of $UV$.  It follows that $x$ appears only inside $s_{13}(V)Up_9(V)$, or only inside $s_{13}(U)Vp_9(U).$  We have four cases, as in Theorem~\ref{preserve}.

\medskip

\noindent
\textbf{Case 1:} $xyx$ is a factor of $s_{13}(V) U p_{9}(V)$.

\noindent
Then $xyx$ is a factor of $\d(v_\ell u v_1)$, and since $\d$ is $\tfrac{7}{4}^+$-free~\cite{Dejean1972}, we reach a contradiction as in Case 1 in the proof of Theorem~\ref{preserve}.

\medskip

\noindent
\textbf{Case 2:} $x$ is a factor of $s_{13}(V)Up_9(V)$ and $xyx$ has $V$ as a factor.

\noindent
The argument is analogous to that of Case 2 in the proof of Theorem~\ref{preserve}.

\medskip

\noindent
\textbf{Case 3:} $xyx$ is a factor of $s_{13}(U) Vp_{9}(U)$.

\noindent
Since $|v|=\ell\leq 20$, it suffices to check that $\dt(w)$ is $\tfrac{7}{4}^+$-free for all $\tfrac{7}{4}^+$-free words $w\in \Sigma^*$ with $|w|\leq 22.$

\medskip

\noindent
\textbf{Case 4:}  $x$ is a factor of $s_{13}(U)V p_9(U)$ and $xyx$ has $U$ as a factor.

\noindent
Write $U=p_9(U)U's_{13}(U).$  Then $y$ contains $U'$ as a factor, so
\[
|y|\geq |U'|=|U|-22=19k-22\geq 92
\]
from the assumption that $k\geq 6$.  On the other hand, since $x$ appears twice (without overlapping itself) in $s_{13}(U)Vp_9(U),$ we conclude that
\[
|x|\leq \frac{|V|+22}{2}=\frac{23\ell+22}{2}\leq 241,
\]
from the assumption that $\ell\leq 20$.  But then $3|y|>|x|,$ a contradiction.
\end{proof}

\begin{theorem}
For every $n\geq 23$, there is a $\tfrac{7}{4}^+$-free circular ternary word of length $n$.
\end{theorem}

\begin{proof}
The proof is by strong induction on $n$.  For $23\leq n< 555$ we found a $\tfrac{7}{4}^+$-free circular ternary word of length $n$ by computer search.

Assume that for some $n\geq 555$, there is a $\tfrac{7}{4}^+$-free circular ternary word of every length $m$ such that $23\leq m<n.$  Then $n-19(6)-23(2)\geq 555-114-46=395,$ so by Lemma~\ref{postage}, we can write $n-19(6)-23(2)=19r+23s,$ or equivalently $n=19(r+6)+23(s+2)$, for integers $r\geq 0$ and $0\leq s \leq 18$.  Let $k=r+6$ and $\ell=s+2,$ and note that $k\geq 6$ and $2\leq \ell\leq 20.$  Now since $555>23^2,$ we have $23\leq k+\ell<n$, so by the inductive hypothesis, there is a $\tfrac{7}{4}^+$-free circular ternary word $(w)$ of length $k+\ell$.  Let $w=uv$, with $|u|=k$ and $|v|=\ell.$  By Theorem~\ref{preserve3}, the circular ternary word $(\d(u)\dt(v))$ is also $\tfrac{7}{4}^+$-free, and has length $19k+23\ell=n$.  
\end{proof}

\section{$\CRTS(5)=\tfrac{4}{3}$}\label{5}

In this section, we prove that the strong circular repetition threshold for $5$ letters is $\tfrac{4}{3}$.  We use a method similar to the method used by Gorbunova \cite{Gorbunova2012} for larger alphabets.  Throughout this section, for a (finite or infinite) word $u=u_1u_2u_3\dots$, where the $u_k$ are letters, we define $u[i]=u_i$ and $u[i:j]=u_i\dots u_j$ for all positive integers $i$ and $j$ with $i\leq j$.  For a finite word $v=v_1v_2\dots v_n$, we define $v[-j]=v[n-j+1]=v_{n-j-1}$.  In other words, $v[-j]$ is the $j$th letter of $w$ counting from the right (and starting at $1$).  For positive integers $i$ and $j$ with $i\geq j$, we define $v[-i:-j]=v[n-i+1:n-j+1]=v_{n-i+1}\dots v_{n-j+1}$.  Finally, for positive integers $i$ and $j$ with $i+j\leq n+1$, we define $v[i:-j]=v_i\dots v_{n-j+1}$.  Throughout this section, words are always assumed to start at index $1$.

Define $\varphi:\{\tt{0},\tt{1}\}^*\rightarrow \{\tt{0},\tt{1}\}^*$ by $\varphi(\tt{0})=\tt{101101}$ and $\varphi(\tt{1})=\tt{10}.$  Define 
\[
\p=\varphi^\infty(\tt{1})=\tt{10\ 101101\ 10\ 101101\ 10\ 10\ 101101\ 10\ 10\ 101101}...
\]

Throughout this section, let $A=\{\tt{a},\tt{b},\tt{c},\tt{d}\}$ and let $s=\tt{abc}.$  For a (finite or infinite) binary word $w$, define
\[
M(w)[i]=\begin{cases}
s[i] \mbox{ if } i\leq 3;\\
M(w)[i-3] \mbox{ if } i>3 \mbox{ and } w[i]=\tt{0};\\
\mbox{the unique element of } \\
A\backslash\{M(w)[i-1],M(w)[i-2],M(w)[i-3]\} \mbox{ otherwise.}
\end{cases}
\]
This definition originates with Pansiot~\cite{Pansiot1984}, though we use different terminology.  The word $w$ is called the \emph{encoding} of $M(w)$.  Note that the encoding $w$ can be recovered from $M(w)$, i.e.\ $M$ is invertible.  Throughout this section, let $\pp=M(\p)$.  The following results are due to Pansiot~\cite{Pansiot1984}.

\begin{result}\label{Pansiot}
\begin{enumerate}
\item \label{TwoZeros} The word $\tt{00}$ is not a factor of $\p$.
\item \label{SevenFifths} The word $\pp$ is $\tfrac{7}{5}^+$-free.
\item \label{BadFactors} The only factors of $\pp$ with exponent greater than $\tfrac{4}{3}$ are encoded by 
\[
t_1=\tt{10110101101}.
\]
This factor of $\p$ gives rise to $\tfrac{7}{5}$-powers in $\pp$ equal to
\[
T_1=M(t_1)=\tt{abcd \ bacbdc \ abcd}
\]
up to permutation of $A$.  
\item \label{AllPermutations} If a finite word $u$ appears as a factor of $\pp,$ then so does $\sigma(u)$, where $\sigma$ is any permutation of $A.$
\end{enumerate}
\end{result}

We also require one further result which can be derived from Pansiot's work fairly easily.  To state the result, we require some additional terminology related to the generalized repetition threshold~\cite{IlieOchemShallit2005}.  For a rational number $\beta$ such that $1<\beta\leq 2$, a $\beta$-power $u$ has \emph{period} $p$ if we can write $u=xx'$, where $x'$ is a prefix of $x$, $|x|=p$, and $|u|=\beta|x|$; in this case, we call $u$ a \emph{$(\beta,p)$-power}, and we call $x'$ the \emph{excess} of the $\beta$-power.  A word is called \emph{$(\beta^+,p)$-free} if it contains no factor that is a $(\beta',p')$-power for $\beta'> \beta$ and $p'\geq p$.

\begin{lemma}\label{PansiotImprovement}
The word $\pp$ is $(\tfrac{14}{11}^+,11)$-free.
\end{lemma}

\begin{proof}
First of all, if $\pp$ has a $(\beta,p)$-power with excess of length at most $2$, and if $p\geq 11$, then $\beta\leq \tfrac{13}{11}$.  So we need only consider powers with excess at least $3$.  The proof is analogous to that of~\cite[Propri\'et\'e 4.14]{Pansiot1984}, so we omit some details.  Define $\mu:\{0,1\}^*\rightarrow \{0,1\}^*$ by $\mu(w)=\varphi(w)\tt{101},$ as in~\cite{Pansiot1984}.  By~\cite[Propri\'et\'e 4.9, 4.10, and 4.12]{Pansiot1984}, it suffices to check that $\mu(uvu)$ and $\mu^2(uvu)$ have exponent less than $\tfrac{14}{11}$ for $u=\varepsilon$ and $v=\tt{101101101}$, and $u=\tt{1}$ and $v=\tt{011010110}$.  This allows us to conclude that the only \emph{inextensible} repetitions of $\pp$ (see~\cite{Pansiot1984} for the precise definition) with exponent greater than $\tfrac{14}{11}$ and excess at least $3$ are
\begin{align*}
&M(\tt{101101101})=\tt{abc\ dbacbd\ abc}, \mbox{ and}\\ 
&M(\tt{10110101101})=\tt{abcd\ bacbdc\ abcd},
\end{align*}
up to permutation of $A$.  These are a $\tfrac{4}{3}$-power of period $9$ and a $\tfrac{7}{5}$-power of period $10$, respectively.  It follows that any factor with exponent greater than $\tfrac{14}{11}$ has period less than $11$.
\end{proof}

We note that the factor 
\begin{center}
\begin{tikzpicture}
\draw (0,0) node {$t_2=\tt{1011010110110101101}$};
\draw [decorate,decoration={brace,amplitude=4pt}]
(-1.4,0.2) -- (0.65,0.2) node [above,black,midway,yshift=5pt]{$t_1$};
\draw [decorate,decoration={brace,amplitude=4pt,mirror}]
(0.1,-0.2) -- (2.15,-0.2) node [below,black,midway,yshift=-5pt]{$t_1$};
\end{tikzpicture}
\end{center}
also appears in $\p$, which encodes the factor
\begin{align}
T_2=M(t_2)=\tt{\overline{abcd}bacb\underline{dc\overline{ab}}\overline{cd}acba\underline{dcab}}\label{Overlap}
\end{align}
up to permutation of the letters (the overlining and underlining is to emphasize the two $\tfrac{7}{5}$-powers).  However, by exhaustive search, this is the only way that two appearances of $t_1$ in $\p$ can overlap (and no appearance of $t_2$ overlaps with a third appearance of $t_1$).  Further, if $t_1vt_1$ is a factor of $\p$, then $|v|\geq 17$.  It follows that any two nonoverlapping $\tfrac{7}{5}$-powers of $\pp$ have a factor of length at least $14$ between them.

We first describe a systematic way of eliminating the $\tfrac{7}{5}$-powers in a finite factor of $\pp$ by introducing a fifth letter $\tt{e}$.  The key is to note that the excess of every $\tfrac{7}{5}$-power in $\pp$ contains every letter from $A$ exactly once.  So we can eliminate the $\tfrac{7}{5}$-powers, and leave repetitions of exponent at most $\tfrac{4}{3}$, by changing the $\tt{d}$ in either the length $4$ prefix or the length $4$ suffix (but not both) of every $\tfrac{7}{5}$-power to an $\tt{e}$.  The algorithm given in Figure~\ref{EtaAlg} accomplishes this, and the output word also satisfies an additional structural property.

\begin{figure}
\begin{lstlisting}
input: $w$, a finite factor of $\pp$ of length $n$
output: $\eta(w)$, a $\tfrac{4}{3}^+$-free word of length $n$ on $A\cup\{\tt{e}\}$ 

for $i$ from $1$ to $n-13$:
   if $w[i:i+3]=w[i+10:i+13]$, then
      if $i\leq n-21$ and $w[i+8:i+11]=w[i+18:i+21]$, then
        if $w[i]=\tt{d}$ or $w[i+1]=\tt{d}$, then
          change the appearance of $\tt{d}$ in $w[i+10:i+11]$ to $\tt{e}$  
        else
          change the appearance of $\tt{d}$ in $w[i:i+3]$ to $\tt{e}$, and
          change the appearance of $\tt{d}$ in $w[i+18:i+21]$ to $\tt{e}$
      else
        if $i< 19$, then
          change the appearance of $\tt{d}$ in $w[i+10:i+13]$ to $\tt{e}$
        if $i\geq 19$, then
          change the appearance of $\tt{d}$ in $w[i:i+3]$ to $\tt{e}$
         
return $w$
\end{lstlisting}
\caption{The defining algorithm for $\eta(w)$.}\label{EtaAlg}
\end{figure}

Essentially, we search through the factors of $w$ of length $14$ for the $\tfrac{7}{5}$-powers.  When we find a $\tfrac{7}{5}$-power $u$, we first check whether it overlaps with another $\tfrac{7}{5}$-power $v$, in which case the factor of $w$ at hand is equal to $T_2$ (see~(\ref{Overlap})) up to permutation of $A$.  If $\tt{d}$ is the first or second letter of $u$, then we change the appearance of $\tt{d}$ that lies in the length $4$ suffix of $u$ (and also the length $4$ prefix of $v$) to an $\tt{e}$.  This one change eliminates both of the overlapping $\tfrac{7}{5}$-powers.  Otherwise, if $\tt{d}$ is the third or fourth letter of $u$, we change the appearance of $\tt{d}$ in the length $4$ prefix of $u$ to an $\tt{e}$ and the appearance of $\tt{d}$ in the length $4$ suffix of $v$ to an $\tt{e}$.  On the other hand, if $u$ does not overlap with another $\tfrac{7}{5}$-power, we consider where $u$ occurs in $w$.  If the $\tfrac{7}{5}$-power at hand starts in the first eighteen letters of $w$, then we replace the appearance of $\tt{d}$ in the length $4$ suffix with an $\tt{e}$.  Otherwise, we replace the appearance of $\tt{d}$ in the length $4$ prefix with an $\tt{e}.$

For a given factor $w$ of $\pp$, the word $\eta(w)$ is not only $\tfrac{4}{3}^+$-free, but any two appearances of $\tt{e}$ in $\eta(w)$ are relatively far apart.  The algorithm is also structured so that the letter $\tt{e}$ does not appear too close to the beginning nor the end of $\eta(w)$ very often; this simply reduces the number of exceptional cases that we need to deal with later.

\begin{lemma}\label{weproperties}
Let $w$ be a factor of $\pp$.  Then the following hold:
\begin{enumerate}
\item \label{FourThirds} $\eta(w)$ is $\tfrac{4}{3}^+$-free.
\item \label{Frequency} Every factor of $\eta(w)$ of length $15$ contains at least two $\tt{d}$'s and at most one $\tt{e}$.
\end{enumerate}
\end{lemma}

\begin{proof}
First we prove~\ref{FourThirds}.  We can recover $w$ from $\eta(w)$ by changing all $\tt{e}$'s to $\tt{d}$'s, so every $\beta$-power in $\eta(w)$ gives rise to a $\beta$-power in $w$.  Thus, by Result~\ref{Pansiot}\ref{SevenFifths}, $\eta(w)$ is $\tfrac{7}{5}^+$-free.  Further, Result~\ref{Pansiot}\ref{BadFactors} describes the factors with exponent greater than $\tfrac{4}{3}$ in $w$.  We change either the prefix of length $4$ or the suffix of length $4$ (but not both) of every $\tfrac{7}{5}$-power in $w$ when we construct $\eta(w)$, so we conclude that $\eta(w)$ is $\tfrac{4}{3}^+$-free.

For~\ref{Frequency}, note first that every factor of length $15$ in $\pp$ contains at least three $\tt{d}$'s (this is by exhaustive check).  It suffices to show that there is a factor of length at least $14$ between any two appearances of $\tt{e}$ in $\eta(w)$.  When constructing $\eta(w)$ from $w$, we only change an appearance of $\tt{d}$ to $\tt{e}$ if it lies inside of a $\tfrac{7}{5}$-power in $w$.  Moreover, we only change a single $\tt{d}$ to an $\tt{e}$ in each $\tfrac{7}{5}$-power.  Take any two distinct appearances of $\tt{e}$ in $\eta(w)$.  If they arise due to overlapping appearances of $\tfrac{7}{5}$-powers in $w$, then the factor of $w$ we are dealing with is equal to $\sigma(T_2)$ for some permutation $\sigma$ of $A$.  Since we are concerned with two distinct appearances of $\tt{e}$, we must have changed some letter in the length $4$ prefix of $\sigma(T_2)$ and the length $4$ suffix of $\sigma(T_2)$.  This leaves a factor of length at least $14$ between the two $\tt{e}$'s.  Otherwise, the two appearances of $\tt{e}$ arise due to nonoverlapping appearances of $\tfrac{7}{5}$-powers in $w$.  But by exhaustive check, there is a factor of length at least $14$ between any pair of nonoverlapping appearances of $\tfrac{7}{5}$-powers in $w$.
\end{proof}

Our constructions of $\tfrac{4}{3}^+$-free circular words on five letters will require factors of $\pp$ that begin and end in the same letter.  The next lemma concerns the existence of such factors.

\begin{lemma}\label{Bookend}
For every $n\geq 3$, $\pp$ either has a factor of length $n$ that begins and ends in $\tt{d}$, or a factor of length $n+1$ that begins and ends in $\tt{d}$.
\end{lemma}

\begin{proof}
Let $n\geq 3$.  First we demonstrate that $\p$ has a factor of length $n$ which begins and ends in $\tt{1}$.  Recall that $\p=\tt{10101101\dots}$, so in particular $\p[5]=\p[6]=\tt{1}.$  By Result~\ref{Pansiot}\ref{TwoZeros}, either $\p[n+4]=\tt{1}$ or $\p[n+5]=\tt{1}$, and hence at least one of the length $n$ factors $\p[5:n+4]$ or $\p[6:n+5]$ begins and ends in $\tt{1}.$

By Result~\ref{Pansiot}\ref{AllPermutations}, it suffices to show that $\pp$ either has a factor of length $n$ that begins and ends in the same letter, or a factor of length $n+1$ that begins and ends in the same letter.  Let $\q$ be a suffix of $\p$ that satisfies $\q[1]=\q[n]=\tt{1}$, and consider the word $\mathbf{Q}=M(\q)$ (note that $\mathbf{Q}$ is a suffix of $\pp$ under a permutation of $A$, but by Result~\ref{Pansiot}\ref{AllPermutations}, every finite factor of $\mathbf{Q}$ is a factor of $\pp$).  Since $\q[1]=\q[n]=\tt{1}$, we see that $\mathbf{Q}[1:4]=\tt{abcd}$, and $\mathbf{Q}[n:n+3]$ contains all $4$ letters of $A$.  If $\mathbf{Q}[n]=\tt{a}$, then $\mathbf{Q}[1:n]$ begins and ends with $\tt{a},$ and we are done.  So we may assume that $\mathbf{Q}[n]\neq \tt{a}$.  If $\mathbf{Q}[n+1]=\tt{a}$, then $\mathbf{Q}[1:n+1]$ begins and ends in $\tt{a}$, and if $\mathbf{Q}[n+1]=\tt{b}$, then $\mathbf{Q}[2:n+1]$ begins and ends in $\tt{b}.$  So we may assume that $\mathbf{Q}[n+1]\in\{\tt{c},\tt{d}\}.$  By similar arguments, we may assume that $\mathbf{Q}[n+2]\in\{\tt{a},\tt{d}\}$ and $\mathbf{Q}[n+3]\in\{\tt{a},\tt{b}\}.$  By inspection, there are only three possibilities for $\mathbf{Q}[n:n+3]$:
\[
\tt{bcda},\ \tt{cdab}, \mbox{ and } \tt{dcab}.
\]

\bigskip

\noindent
\textbf{Case I:} $\mathbf{Q}[n:n+3]=\tt{bcda}$ (see Table~\ref{CaseOneTable})

\medskip

\begin{table}
\centering{
\begin{tabular}{c c c c c c}
\midrule
 & $\qq[k-3]$ & $\qq[k-2]$ & $\qq[k-1]$ & $\qq[k]$ & \\ 
& $\tt{\alpha}$ & $\tt{\beta}$ & $\tt{\gamma}$ \\\midrule

 $\qq[n+k-5]$ & $\qq[n+k-4]$ & $\qq[n+k-3]$ & $\qq[n+k-2]$ & $\qq[n+k-1]$ & $\qq[n+k]$\\
 $\tt{\alpha}$ & $\tt{\beta}$ & $\tt{\gamma}$ & &\\\midrule
\end{tabular}
}
\caption{Visual aid for Case I of the proof of Lemma~\ref{Bookend}.  Columns show first and last letter of factors of length $n$ in $\qq$.}
\label{CaseOneTable}
\end{table}

\noindent
Note that $\mathbf{Q}[n:n+2]=\tt{bcd}=\mathbf{Q}[2:4].$  However, $\qq[2:n-1]$ and $\qq[n:2n-3]$ must differ at some point, as otherwise $\qq[2:2n-3]$ is a square, contradicting Result~\ref{Pansiot}\ref{SevenFifths}.  Let $k$ be the smallest number greater than $2$ such that $\qq[k]\neq \qq[n+k-2].$  Let $\qq[k-3:k-1]=\qq[n+k-5:n+k-3]=\tt{\tt{\alpha\beta\gamma}}$, where $\tt{\alpha},$ $\tt{\beta},$ $\tt{\gamma},$ and $\tt{\delta}$ are the images of $\tt{a},$ $\tt{b},$ $\tt{c},$ and $\tt{d}$, respectively, under some permutation of $A$.  If $\qq[k]=\tt{\delta}$ and $\qq[n+k-2]=\tt{\alpha},$ then it follows that $\qq[n+k-1]=\tt{\delta}.$  But then $\qq[k:n+k-1]$ has length $n$ and begins and ends with $\tt{\delta}.$  On the other hand, if $\qq[k]=\tt{\alpha}$ and $\qq[n+k-2]=\tt{\delta},$ then either $\qq[n+k-1]=\tt{\alpha}$, or $\qq[n+k]=\tt{\alpha}$.  But then $\qq[k:n+k-1]$, or $\qq[k:n+k]$, respectively, begins and ends with $\tt{\alpha}$.

\bigskip

\noindent
\textbf{Case II:} $\mathbf{Q}[n:n+3]\in\{\tt{cdab},\tt{dcab}\}$ (see Table~\ref{CaseTwoTable})

\medskip

\begin{table}
\centering{
\begin{tabular}{c c c c c c}
\midrule
$\qq[k-3]$ & $\qq[k-2]$ & $\qq[k-1]$ & $\qq[k]$ & $\qq[k+1]$ & $\qq[k+2]$ \\ 
 $\tt{\alpha}$ & $\tt{\beta}$ & $\tt{\gamma}$ \\\midrule

& & $\qq[n+k-2]$ & $\qq[n+k-1]$ & $\qq[n+k]$ & $\qq[n+k+1]$\\
 & & $\tt{\alpha}$ & $\tt{\beta}$ & $\tt{\gamma}$ & \\\midrule
\end{tabular}
}
\caption{Visual aid for Case II of the proof of Lemma~\ref{Bookend}.  Columns show first and last letter of factors of length $n$ in $\qq$.}
\label{CaseTwoTable}
\end{table}

\noindent
If $\qq[n+4]=\tt{d},$ then $\qq[4:n+4]$ begins and ends in $\tt{d},$ and we are done, so we may assume that $\qq[n+4]=\tt{c}.$  Then $\mathbf{Q}[1:3]=\tt{abc}=\mathbf{Q}[n+2:n+4]$.  However, $\qq[1:n+1]$ and $\qq[n+2:2n+2]$ must differ at some point, as otherwise $\qq[1:2n+2]$ is a square, contradicting Result~\ref{Pansiot}\ref{SevenFifths}.  Let $k$ be the smallest number such that $\qq[k]\neq \qq[n+k+1]$.  Let $\qq[k-3:k-1]=\qq[n+k-2:n+k]=\tt{\alpha\beta\gamma},$ where $\tt{\alpha},$ $\tt{\beta},$ $\tt{\gamma},$ and$\tt{\delta}$ are the images of $\tt{a},$ $\tt{b},$ $\tt{c},$ and $\tt{d}$, respectively, under some permutation of $A$.  If $\qq[k]=\tt{\alpha}$ and $\qq[n+k+1]=\tt{\delta},$ then it follows that $\qq[k+1]=\tt{\delta}$, and $\qq[k+1:n+k+1]$ begins and ends with $\tt{\delta}$.  Otherwise, $\qq[k]=\tt{\delta}$ and $\qq[n+k+1]=\tt{\alpha}.$  Then either $\qq[k+1]=\tt{\alpha}$, or $\qq[k+2]=\tt{\alpha}.$  But then either $\qq[k+1:n+k+1]$, or $\qq[k+2:n+k+1]$, respectively, begins and ends in $\tt{\alpha}.$
\end{proof}

We believe that a stronger version of Lemma~\ref{Bookend} holds.  We conjecture that for every $n\geq 4$, there is a factor of length $n$ in $\pp$ that begins and ends in the same letter.  We have verified this statement for $4\leq n\leq 2000.$  A proof of this fact would reduce the amount of case work required in the proof of the main result of this section, which we are now ready to prove.

We first show that there is a $\tfrac{4}{3}^+$-free circular word on $5$ letters of every even length.  We then construct words of every odd length from the words of even length by inserting a single letter $\tt{e}$ in a carefully chosen location.

\begin{theorem}\label{Even}
For every $n\in\mathbb{N}$, there is a $\tfrac{4}{3}^+$-free circular word on $5$ letters of length $2n$.
\end{theorem}

\begin{proof}
For every $n< 73$, we find such a word by computer search, so we may assume that $n\geq 73$.  By Lemma~\ref{Bookend}, $\pp$ either has a factor of length $n-4$ that begins and ends in $\tt{d}$, or a factor of length $n-5$ that begins and ends in $\tt{d}.$  This gives us two cases:

\bigskip

\noindent
\textbf{Case I:} $\pp$ has a factor $w'$ of length $n-4$ that begins and ends in $\tt{d}$

\medskip

Let $w=\alpha w' \beta$ be a factor of $\pp$ containing $w'$, where $\alpha,\beta\in A$.  By Result~\ref{Pansiot}\ref{AllPermutations}, we may assume that $\alpha=\tt{a}$, and $\beta\in\{\tt{a},\tt{b}\}.$  Let $\sigma$ be the permutation of $A\cup\{\tt{e}\}$ defined by $(\tt{a},\tt{b},\tt{c})(\tt{d},\tt{e})$ in cycle notation.  Consider the word
\[
W_{2n}=\tt{de} \ \eta(w) \ \tt{ed} \ [\sigma(\eta(w))]^R.
\]
We claim that the circular word $(W_{2n})$ is $\tfrac{4}{3}^+$-free, except in a small number of exceptional cases which are handled later by making minor adjustments to $\eta(w)$.

Suppose towards a contradiction that some conjugate of $W_{2n}$ has a factor $xyx$ with $2|x|>|y|.$  First of all, note that $x$ does not contain the factor $\tt{de}$ or the factor $\tt{ed}$, since each of these factors appears at most once in any conjugate of $W_{2n}$.  So we may assume that $x$ is a factor of either $\tt{e}\eta(w)\tt{e}$ or $\tt{d}[\sigma(\eta(w))]^R\tt{d}.$
Further, note that the length $3$ prefix and the length $3$ suffix of $\tt{e}\eta(w)\tt{e}$, namely $\tt{ead}$ and $\tt{d\beta e},$ respectively,  and the length $3$ prefix and the length $3$ suffix of $\tt{d}[\sigma(\eta(w))]^R\tt{d}$, namely $\tt{d\sigma(\beta)e}$ and $\tt{e\sigma(\tt{a})d},$ respectively, appear at most once in any conjugate of $W_{2n}$.  This is independent of the identity of $\beta$, and simply relies on the fact that $\sigma(\tt{a})\neq \tt{a}$ and $\sigma(\tt{b})\neq \tt{b}$.  Therefore, $x$ does not contain any of these factors.  The possibility that $x$ appears as a prefix or suffix of $\tt{e}\eta(w)\tt{e}$ or $\tt{d}\sigma(\eta(w))^R\tt{d}$ with $|x|=2$ is eliminated by a later exhaustive search, so we may assume that $x$ is a factor of either $\eta(w)$ or $\sigma(\eta(w))^R$.

Suppose first that the factor $xyx$ appears in $(W_{2n})$ so that both appearances of $x$ lie in $\eta(w)$.  By Lemma~\ref{weproperties}\ref{FourThirds}, $xyx$ is not a factor of $\eta(w)$.  On the other hand, if $y$ contains all of $\tt{ed}[\sigma(\eta(w))]^R\tt{de}$, then $|y|\geq n+2,$ which is over half the length of $W_{2n}$, and contradicts the assumption that $2|x|>|y|$.  So the factor $xyx$ does not appear in such a way that both appearances of $x$ lie in $\eta(w)$.  By a similar argument, the factor $xyx$ does not appear in such a way that both appearances of $x$ lie in $\sigma(\eta(w))^R$ (note that $\sigma(\eta(w))^R$ is also $\tfrac{4}{3}^+$-free, because this property is preserved under permutation of the letters and reversal).

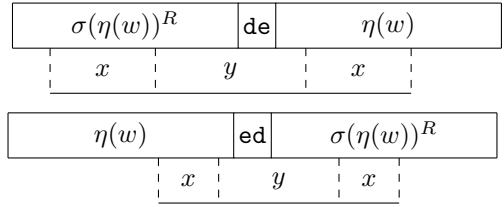
\begin{figure}

%
%

\centering{
\begin{tikzpicture}
\draw (3.5,0) rectangle (4,0.6) node[pos=.5] {$\tt{de}$};
\draw (4,0) -- (7,0);
\draw (7,0) -- (7,0.6);
\draw (4,0.6) -- (7,0.6);
\draw (0.5,0) -- (3.5,0);
\draw (0.5,0) -- (0.5,0.6);
\draw (0.5,0.6) -- (3.5,0.6);

\draw[dashed] (1,0) -- (1,-0.6);
\draw[dashed] (2.4,0) -- (2.4,-0.6);
\draw[dashed] (4.4,0) -- (4.4,-0.6);
\draw[dashed] (5.8,0) -- (5.8,-0.6);

\draw (2,0.3) node {$\sigma(\eta(w))^R$};
\draw (5.5,0.3) node {$\eta(w)$};
\draw (1.7,-0.3) node {$x$};
\draw (3.4,-0.3) node {$y$};
\draw (5.1,-0.3) node {$x$};
\draw (1,-0.6) -- (5.8,-0.6);
\end{tikzpicture}\\

\medskip

\begin{tikzpicture}
\draw (3.5,0) rectangle (4,0.6) node[pos=.5] {$\tt{ed}$};
\draw (4,0) -- (7,0);
\draw (7,0) -- (7,0.6);
\draw (4,0.6) -- (7,0.6);
\draw (0.5,0) -- (3.5,0);
\draw (0.5,0) -- (0.5,0.6);
\draw (0.5,0.6) -- (3.5,0.6);

\draw[dashed] (2.5,0) -- (2.5,-0.6);
\draw[dashed] (3.3,0) -- (3.3,-0.6);
\draw[dashed] (4.9,0) -- (4.9,-0.6);
\draw[dashed] (5.7,0) -- (5.7,-0.6);

\draw (2,0.3) node {$\eta(w)$};
\draw (5.5,0.3) node {$\sigma(\eta(w))^R$};
\draw (2.9,-0.3) node {$x$};
\draw (4.1,-0.3) node {$y$};
\draw (5.3,-0.3) node {$x$};
\draw (2.5,-0.6) -- (5.7,-0.6);
\end{tikzpicture}
}
\caption{Possible appearances of $xyx$ in $(W_{2n})$.}
\label{Appearances}
\end{figure}

So we may assume that $xyx$ appears in $(W_{2n})$ in such a way that one appearance of $x$ is in $\eta(w)$ and the other is in $\sigma(\eta(w))^R$.  By Lemma~\ref{weproperties}\ref{Frequency}, every factor of length $15$ in $\eta(w)$ contains at least two $\tt{d}$'s and at most one $\tt{e}$.  Since $\sigma$ swaps $\tt{d}$ and $\tt{e}$, every factor of length $15$ of $\sigma(\eta(w))^R$ contains at least two $\tt{e}$'s and at most one $\tt{d}$.  Since $x$ must appear in both $\eta(w)$ and $\sigma(\eta(w))^R$, we must have $|x|\leq 14$.

The remainder of the proof is completed by a finite search.  Since $|x|\leq 14$ and $|y|< 2|x|$, we must have $|xyx|< 56$.  Further, since $xyx$ must appear as in Figure~\ref{Appearances}, we only need to search a short factor of $\sigma(\eta(w))^R \tt{de}\eta(w)$ (and $\eta(w)\tt{ed}\sigma(\eta(w))^R$) around $\tt{de}$ ($\tt{ed}$, respectively).  In fact, since the two appearances of $x$ must lie on opposite sides of $\tt{de}$ ($\tt{ed}$, respectively), $xyx$ can extend at most $39$ letters away from $\tt{de}$ ($\tt{ed}$, respectively).

So it suffices to check that $\sigma(p_{39}(\eta(w)))^R \tt{de}p_{39}(\eta(w))$ and $s_{39}(\eta(w))\tt{ed}\sigma(s_{39}(\eta(w)))^R$ are $\tfrac{4}{3}^+$-free.  We know that $\eta(w)$ begins in $\tt{ad}$ and ends in either $\tt{da}$ or $\tt{db}$, but we can't assume anything else about the prefix or the suffix of $\eta(w)$; we just know that $w$ is a factor of $\pp$.  The length $39$ prefix of $\eta(w)$ is completely determined by the length $60$ prefix of $w$.  The length $60$ prefix of $w$ is required because factors of length $14$ (nonoverlapping $\tfrac{7}{5}$-powers) and $22$ (overlapping $\tfrac{7}{5}$-powers) determine whether or not we change a given $\tt{d}$ in $w$ to an $\tt{e}$ in $\eta(w)$.  Similarly, the length $39$ suffix of $\eta(w)$ is completely determined by the length $60$ suffix of $w$.  This is where the assumption that $n\geq 73$ is used, so that we may assume that all nonoverlapping $\tfrac{7}{5}$-powers in $s_{52}(w)$ (i.e.\ all those that could impact $s_{39}(\eta(w))$) are eliminated in $\eta(w)$ by changing a $\tt{d}$ to an $\tt{e}$ in the prefix (and not the suffix) of the $\tfrac{7}{5}$-power.

We run through all possible prefixes $p_{39}(\eta(w))$, and find that
\[
\sigma(p_{39}(\eta(w)))^R \tt{de}p_{39}(\eta(w))
\]
is $\tfrac{4}{3}^+$-free in all but a small number of cases.  In each of these exceptional cases, we make a minor adjustment in a short prefix of $\eta(w)$ to fix the issue; see Appendix~\ref{Appendix1} for details.  Similarly, we run through all possible suffixes $s_{39}(\eta(w))$, and find that
\[
s_{39}(\eta(w))\tt{ed}\sigma(s_{39}(\eta(w)))^R
\]
is $\tfrac{4}{3}^+$-free in all but a small number of cases.  In each of these exceptional cases, we make a minor adjustment in a short suffix of $\eta(w)$ to fix the issue; see Appendix~\ref{Appendix1} for details.

\bigskip

\noindent
\textbf{Case II:} $\pp$ has a factor $w'$ of length $n-5$ that begins and ends in $\tt{d}$

\medskip

Let $w=\alpha\beta w' \gamma$ be a factor of $\pp$ containing $w'$, where $\alpha,\beta,\gamma\in A$.  By Result~\ref{Pansiot}\ref{AllPermutations}, we may assume that $\alpha=\tt{a}$, $\beta=\tt{b},$ and $\gamma\in\{\tt{a},\tt{b},\tt{c}\}$.  We have two subcases:

\bigskip

\noindent
\textbf{Case II(a):} $\gamma\in\{\tt{a},\tt{c}\}$.  

\medskip

Let $\tau$ be the permutation of $A\cup\{\tt{e}\}$ defined by $(\tt{a},\tt{c})(\tt{d},\tt{e})$ in cycle notation (note that $\tt{b}$ is fixed by $\tau$).  Consider the word
\[
W_{2n}=\tt{de} \ \eta(w) \ \tt{ed} \ [\tau(\eta(w))]^R.
\]
We claim that the circular word $(W_{2n})$ is $\tfrac{4}{3}^+$-free, except in a small number of exceptional cases which are handled later by making minor adjustments to $\eta(w)$.

Suppose towards a contradiction that some conjugate of $W$ has a factor $xyx$ with $2|x|>|y|.$  As in Case I, $x$ does not contain the factor $\tt{de}$ or the factor $\tt{ed}$.  So we may assume that $x$ is a factor of either $\tt{e}\eta(w)\tt{e}$ or $\tt{d}[\tau(\eta(w))]^R\tt{d}.$
Further, note that the length $3$ suffix of $\tt{e}\eta(w)\tt{e}$, namely $\tt{d\gamma e},$ and the length $3$ prefix of $\tt{d}[\tau(\eta(w))]^R\tt{d}$, namely $\tt{d\tau(\gamma)e}$ appear only once in $(W_{2n})$.  Therefore, $x$ does not contain either of these factors, and we may assume that $x$ is a factor of either $\tt{e}\eta(w)$ or $\tau(\eta(w))^R\tt{d}$.  Note that the length $4$ prefix of $\tt{e}\eta(w)$, namely $\tt{eabd}$, and the length $4$ suffix of $\tau(\eta(w))^R\tt{d}$, namely $\tt{ebcd}$, may both appear elsewhere in $(W_{2n})$, so we cannot immediately eliminate the possibility that $x$ contains one of these factors (as we did in Case I).

Suppose first that the factor $xyx$ appears in $(W_{2n})$ so that both appearances of $x$ lie in $\tt{e}\eta(w)$.  By Lemma~\ref{weproperties}\ref{FourThirds}, $xyx$ is not a factor of $\eta(w)$.  Further, it cannot be the case that $y$ contains all of $\tt{ed}[\tau(\eta(w))]^R\tt{d}$, because then $|y|\geq n+1,$ which is over half the length of $W_{2n}$.  The only possibility that remains is that $xyx$ is a prefix of $\tt{e}\eta(w)$ (note that $\tt{d}w$ is not necessarily a factor of $\pp$ so we cannot apply Lemma~\ref{weproperties}\ref{FourThirds} directly here).  By exhaustive search of the possible prefixes of $\eta(w)$ of length $19$, we may assume that $|x|\geq 6$ and $|xy|\geq 11$.  Write $x=\tt{e}x'$, so that $xyx=\tt{e}x'y\tt{e}x'$.    We see that $x'y\tt{e}x'$ is a repetition of exponent $\frac{|xyx|-1}{|xy|}$.  Since $|xy|\geq 11$, by Lemma~\ref{PansiotImprovement}, we must have 
\begin{align*}
\frac{|xyx|-1}{|xy|}\leq \tfrac{14}{11} \ \ \ 
&\Rightarrow \ \ \ 8|x|-11\leq 3|y|.
\end{align*}
Since $|x|\geq 6,$ we have $8|x|-11\geq 6|x|+12-11>6|x|,$ and thus $2|x|<|y|,$ a contradiction.  By a similar argument, the factor $xyx$ cannot appear in such a way that both appearances of $x$ lie in $\tau(\eta(w))^R\tt{d}$.

So we may assume that $xyx$ appears in $(W_{2n})$ in such a way that one appearance of $x$ is in $\tt{e}\eta(w)$ and the other is in $\tau(\eta(w))^R\tt{d}$.  By Lemma~\ref{weproperties}\ref{Frequency}, every factor of $\eta(w)$ of length $15$ contains at least two $\tt{d}$'s and at most one $\tt{e}$.  Since $\tau$ swaps $\tt{d}$ and $\tt{e}$, every factor of length $15$ of $\tau(\eta(w))^R$ contains at least two $\tt{e}$'s and at most one $\tt{d}$.  

Suppose that $xyx$ appears so that one appearance of $x$ is a prefix of $\tt{e}\eta(w)$ and $|x|\geq 17$.  Then $x[2:16]$ is a factor of $\eta(w)$ of length $15$, and hence is not a factor of $\tau(\eta(w))^R$.  But then $x$ is not a factor of $\tau(\eta(w))^R\tt{d}$.   If $x$ is a prefix of $\tt{e}\eta(w)$ and $|x|\leq 16$, then certainly $xyx$ appears in $\tau(\eta(w))^R\tt{de}\eta(w)$ and not $\tt{e}\eta(w)\tt{ed}\tau(\eta(w))^R\tt{d},$ since $n\geq 73.$  Since $|x|\leq 16$, $xyx$ must in fact appear in $\tau(\eta(w)[1:46])^R\tt{de}\eta(w)[1:15]$.  We eliminate this possibility by exhaustively checking the possible prefixes of $\eta(w)$.  By a similar argument, we may assume that $x$ is not a suffix of $\tau(\eta(w))^R\tt{d}$.   

But then $xyx$ appears in $(W_{2n})$ in such a way that one appearance of $x$ is in $\eta(w)$ and the other is in $\tau(\eta(w))^R$.  We then have $|x|\leq 14$, and the proof is completed by an exhaustive search as in Case I.  Here we check that
\[
\tau(p_{46}(\eta(w)))^R \, \tt{de}\, p_{46}(\eta(w))
\] 
is $\tfrac{4}{3}^+$-free (we check a longer prefix than in Case I as this is required in the previous paragraph), and that
\[
s_{39}(\eta(w))\, \tt{ed}\, \tau(s_{39}(\eta(w)))^R
\]
is $\tfrac{4}{3}^+$-free.  Once again, there are several exceptional cases in which we need to make a minor adjustment in a short prefix and/or suffix of $\eta(w)$; see Appendix~\ref{Appendix2} for details.

\bigskip

\noindent
\textbf{Case II(b):} $\gamma=\tt{b}$.

\medskip

Here, we have $w=\tt{ab} w' \tt{b}$, and recall that $|w|=n-2$.  While we would have liked to find a way to handle this case as we did for Case I and Case II(a), we must do something slightly different here to avoid repetitions of short length close to the `buffers' $\tt{de}$ and $\tt{ed}$.  Let $\delta\in A$ be a letter such that $w\delta$ is a factor of $\pp$.  Since $w$ ends in $\tt{db},$ we must have $\delta\in\{a,c\}$, meaning two cases, though the second reduces to an already completed case.

\bigskip

\noindent
\textbf{Case II(b1):} $\delta=\tt{a}$

\medskip

Let $\dot{w}=\tt{b}w'\tt{b}$ and $\ddot{w}=\tt{ab}w'\tt{ba}.$  Note that $|\dot{w}|=n-3$ and $|\ddot{w}|=n-1$.  Let $\pi$ be the permutation of $A$ defined by $(\tt{b},\tt{d},\tt{c})$ and let $\rho$ be the permutation of $A\cup \{\tt{e}\}$ defined by $(\tt{d},\tt{e})$ (i.e.\ $\rho$ swaps $\tt{d}$ and $\tt{e}$).  Define
\[
W_{2n}=\tt{de} \ \eta(\dot{w}) \ \tt{ed} \ \rho\left(\eta(\pi(\ddot{w}))\right)^R.
\]
We claim that the circular word $(W_{2n})$ has no $\tfrac{4}{3}^+$-powers, except in a small number of cases which are handled later by making small adjustments to $\eta(\dot{w})$ and/or $\rho(\eta(\pi(\ddot{w})))^R$.  Most of the proof is similar to that of Case I, so we omit some details.

Suppose towards a contradiction that some conjugate of $W_{2n}$ has a factor $xyx$ with $2|x|>|y|.$  First of all, note that $x$ does not contain the factor $\tt{de}$ or the factor $\tt{ed}$, so we may assume that $x$ is a factor of either $\tt{e}\dot{w}_{\tt{e}}\tt{e}$ or $\tt{d}\rho\left([\pi(\ddot{w})]_{\tt{e}}\right)^R\tt{d}.$
Further, note that the length $3$ prefix and the length $3$ suffix of $\tt{e}\eta(\dot{w})\tt{e}$, namely $\tt{ebd}$ and $\tt{dbe},$ respectively,  and the length $3$ prefix and the length $3$ suffix of $\tt{d}\rho(\eta(\pi(\ddot{w})))^R\tt{d}$, namely $\tt{dae}$ and $\tt{ead},$ respectively, appear at most once in any conjugate of $W_{2n}$.  Therefore, $x$ does not contain any of these factors, and we may assume that $x$ is a factor of either $\eta(\dot{w})$ or $\rho(\eta(\pi(\ddot{w})))^R$.

By an argument similar to the one used in Case I, the factor $xyx$ cannot appear in such a way that both appearances of $x$ lie in $\eta(\dot{w})$ (or $\rho(\eta(\pi(\ddot{w})))^R$).  So we may assume that $xyx$ appears in $(W_{2n})$ in such a way that one appearance of $x$ is in $\eta(\dot{w})$ and the other is in $\rho(\eta(\pi(\ddot{w})))^R$.  By Lemma~\ref{weproperties}\ref{Frequency}, every factor of length $15$ in $\eta(\dot{w})$ or $\eta(\pi(\ddot{w}))$ contains at least two $\tt{d}$'s and at most one $\tt{e}$.  Since $\rho$ swaps $\tt{d}$ and $\tt{e}$, every factor of length $15$ of $\rho(\eta(\pi(\ddot{w})))^R$ contains at least two $\tt{e}$'s and at most one $\tt{d}$.  Since $x$ must appear in both $\eta(\dot{w})$ and $\rho(\eta(\pi(\ddot{w})))^R$, we must have $|x|\leq 14$.

The remainder of the proof is completed by a finite search, as in Case I.  It suffices to check that 
\[
\rho(p_{39}(\eta(\pi(\ddot{w}))))^R \tt{de}p_{39}(\eta(\dot{w}))
\]
and
\[
s_{39}(\eta(\dot{w}))\tt{ed}\rho(s_{39}(\eta(\pi(\ddot{w}))))^R
\]
are $\tfrac{4}{3}^+$-free.  As in Case I and Case II(a), there are several exceptional cases in which we need to make a minor adjustment in a short prefix and/or suffix of $\eta(\dot{w})$ or $\eta(\pi(\ddot{w}))$; see Appendix~\ref{Appendix3} for details.

\bigskip

\noindent
\textbf{Case II(b2):} $\delta=\tt{c}$

\medskip

We show that $\pp$ either has a factor of length $n-4$ that begins and ends in $\tt{d}$ (and we are back in Case I), or has a factor $z$ of length $n-5$ that begins and ends in $\tt{d}$ and appears internally as $\tt{ab}z\tt{ba}$, i.e.\ a factor of length $n-1$ that begins in $\tt{abd}$ and ends in $\tt{dba}$ (and we are back in Case II(b1)).

Suppose otherwise that $\pp$ has no factors of either of these forms.  By Lemma~\ref{Pansiot}\ref{AllPermutations}, $\pp$ also does not contain a factor of either of these forms under any permutation of $A$.  Let $k\in\mathbb{N}$ satisfy $\pp[k:n+k-2]=\tt{ab}w'\tt{bc}$.   For ease of reading, we provide Table~\ref{LastCaseTable}.  The entries in black are known initially, while the entries in red are determined by the following arguments.  First of all, note that $\pp[n+k-5]=\tt{c}$, since otherwise $\pp[k:n+k-5]$ is a factor of length $n-4$ that begins and ends in $\tt{a}$.  Hence, $\pp[n+k-2]$ is encoded by $\tt{0}$, and by Lemma~\ref{Pansiot}\ref{TwoZeros}, $\pp[n+k-1]$ must be encoded by $\tt{1}$.  Hence, $\pp[n+k-1]=\tt{a}.$  Now we see that $\pp[k+3]=\tt{a}$, since otherwise $\pp[k+3:n+k-2]$ begins and ends in $\tt{c}$, and the encoding for $\pp[k+3]$ is $\tt{0}$.  Hence, the encoding for $\pp[k+4]$ is $\tt{1}$, and $\pp[k+4]=\tt{c}$.

Now if $\pp[n+k]=\tt{d}$, then the factor $\pp[k+2:n+k]$ under the permutation $(\tt{a},\tt{b},\tt{c},\tt{d})$ puts us back in Case II(b1), so we may assume that $\pp[n+k]=\tt{b}$, and since this is encoded by $\tt{0}$, the next letter in the encoding is $\tt{1}$, and hence $\pp[n+k+1]=\tt{d}$.  Now $\pp[k+5]$ cannot be $\tt{b}$, so it must be $\tt{d}$.  Hence, $\pp[k+5]$ is encoded by $\tt{0}$ and it follows that $\pp[k+6]$ is encoded by $\tt{1}$, so that $\pp[k+6]=\tt{b}$.

Now if $\pp[n+k+2]=\tt{c}$, then the factor $\pp[k+4:n+k+2]$ under the permutation $(\tt{a},\tt{c})(\tt{b},\tt{d})$ puts us back in Case II(b1), so we may assume that $\pp[n+k+2]=\tt{a}$.  Then $\pp[n+k+2]$ is encoded by $\tt{0}$, and $\pp[n+k+3]$ is encoded by $\tt{1}$, so $\pp[n+k+3]=\tt{c}$.  Now $\pp[k+7]$ cannot be $\tt{a}$, so it must be $\tt{c}$.  Hence $\pp[k+7]$ is encoded by $\tt{0}$ and $\pp[k+8]=\tt{a}$ is encoded by $\tt{1}$.  Finally, if $\pp[n+k+4]=\tt{b}$, then the factor $\pp[k+6:n+k+4]$ under the permutation $(\tt{a},\tt{d},\tt{c},\tt{b})$ puts us back in Case II(b1), so we may assume that $\pp[n+k+4]=\tt{d}$, and hence is encoded by $\tt{0}.$  However, this is impossible, because $\p$ does not have $\tt{0101010}$ as a factor.
\end{proof}

\begin{table}
\centering
\resizebox{\columnwidth}{!}{%
\begin{tabular}{c | c c c c c c c c c c}
\midrule
& $\pp[k]$ &  $\pp[k+1]$ & $\pp[k+2]$ & $\pp[k+3]$ & $\pp[k+4]$ & $\pp[k+5]$ & $\pp[k+6]$ & $\pp[k+7]$ & $\pp[k+8]$  \\
&  $\tt{a}$ & $\tt{b}$  & $\tt{d}$ & 
\\ 
encoding & & & & 
\\\midrule 
& $\pp[n+k-5]$ & $\pp[n+k-4]$ & $\pp[n+k-3]$ & $\pp[n+k-2]$ & $\pp[n+k-1]$ & $\pp[n+k]$ & $\pp[n+k+1]$ & $\pp[n+k+2]$ & $\pp[n+k+3]$ & $\pp[n+k+4]$\\
& 
& $\tt{d}$  & $\tt{b}$  & $\tt{c}$ & 
\\
encoding & & & & 
\\\midrule
\end{tabular}
}
\caption{Visual aid for Case II(b2) of the proof of Theorem~\ref{Even}.  Columns contain the first and last letter of factors of length $n-4$ in $\pp$.  The reader can complete the table as they read the proof.}
\label{LastCaseTable}
\end{table}

\begin{theorem}\label{Odd}
For every $n\geq 0$, there is a $\tfrac{4}{3}^+$-free circular word on $5$ letters of length $2n+1$.
\end{theorem}

\begin{proof}
We first verify the statement directly for $n\leq 644$ by computer.  Now suppose $n\geq 645.$  Take the $\tfrac{4}{3}^+$-free circular word $(W_{2n})$ on $A\cup \{e\}$ of length $2n$ constructed in Theorem~\ref{Even}.  Note that $W_{2n}$ may have been constructed with one of the adjustments outlined in Appendix~\ref{Appendix}; our arguments still apply in each of these exceptional cases.  In particular, $W_{2n}$ contains the factor $\eta(w)[23:-23]$ (Case I or Case II(a)) or $\eta(\dot{w})[23:-23]$ (Case II(b1)), where $w$ is a factor of length $n-2\geq 643$ of $\pp$ and $\dot{w}$ is a factor of length $n-3\geq 642$ of $\pp$.

By exhaustive search, if $u$ is any factor of $\pp$ of length $642,$ then $\eta(u)[23:-23]$ contains the factor
\begin{align*}
z=\gamma\beta\alpha\gamma\tt{d}\alpha\beta\tt{d}\gamma \alpha\tt{d} \beta \gamma \tt{d} \alpha\gamma\beta\alpha
\end{align*}
where $\alpha,$ $\beta,$ and $\gamma$ are the images of $\tt{a},$ $\tt{b},$ and $\tt{c}$ under some permutation of $\{\tt{a},\tt{b},\tt{c}\}$.  Note that we ignore the prefix of length $22$ and the suffix of length $22$ of $\eta(u)$ to ensure that this property extends to all factors of $\pp$ of length greater than $642$.  Note that $z$ contains the factors $\tt{d}\alpha\beta \tt{d}$, $\tt{d}\gamma\alpha\tt{d}$, and $\tt{d}\beta\gamma\tt{d}$.  

Now for concreteness, assume that $W_{2n}$ was created using the construction of Case I of Theorem~\ref{Even}.  The other cases are handled by similar arguments.  Recall that
\[
W_{2n}=\tt{de}\,\eta(w)\,\tt{ed},\sigma(\eta(w))^R,
\]
and that $\eta(w)$ begins in $\tt{ad}$ and ends in $\tt{d} \lambda$ for some letter $\lambda\in\{\tt{a},\tt{b}\}$.  Hence, $\sigma(\eta(w))^R$ begins in $\sigma(\lambda \tt{d})=\sigma(\lambda)\tt{e}$ and ends in $\sigma(\tt{da})=\tt{eb}.$  So the `buffers' $\tt{de}$ and $\tt{ed}$ appear internally as
\[
\tt{ebdead} \ \ \mbox{ and } \ \ \tt{d}\lambda\tt{ed}\sigma(\lambda)\tt{e},
\]
respectively.  By construction, these factors contain the only appearances of factors of the form $\tt{d}\kappa \tt{e}$ or $\tt{e}\kappa \tt{d}$ in $(W_{2n})$, where $\kappa$ is a single letter.  Exactly one letter from $\{\alpha,\beta,\gamma\}$ is equal to $\lambda$, and likewise for $\tt{a}$, so at least one of the following must be true:

\begin{enumerate}
\item $\alpha\neq \lambda$ and $\beta\neq \tt{a}$;
\item $\gamma\neq \lambda$ and $\alpha\neq \tt{a}$; or
\item $\beta\neq \lambda$ and $\gamma\neq \tt{a}$.
\end{enumerate}

We obtain a word $W_{2n+1}$ of length $2n+1$ by inserting a single $\tt{e}$ into $W_{2n}$, and claim that $(W_{2n+1})$ is $\tfrac{4}{3}^+$-free.

If (a) is true, insert an $\tt{e}$ into the first appearance of $z$ in $\eta(w)[23:-23]$ as follows:
\[
\gamma\beta\alpha\gamma\tt{d}\alpha\, \tt{e}\, \beta\tt{d}\gamma \alpha \tt{d}\beta \gamma\tt{d}\alpha\gamma\beta\alpha.
\]
Suppose towards a contradiction that $(W_{2n+1})$ contains a factor $xyx$ with $2|x|>|y|$.  Since $(W_{2n})$ is $\tfrac{4}{3}^+$-free and we inserted a single letter $\tt{e}$ into $W_{2n}$ to create $W_{2n+1}$, it must be the case that the new letter $\tt{e}$ appears inside one of the instances of $x$.  We first claim that $x$ cannot contain either of the factors $\tt{d}\alpha \tt{e}$ or $\tt{e}\beta\tt{d}$.  Suppose first that $x$ contains $\tt{d}\alpha\tt{e}$ as a factor.  Then certainly $\tt{d}\alpha\tt{e}$ must appear twice in some conjugate of $W_{2n+1}$.  Since $\alpha\neq \lambda$, we must have $\alpha=\sigma(\lambda)$.  However, the factor $\tt{d}\alpha\tt{e}=\tt{d}\sigma(\lambda)\tt{e}$ appears only inside the factors $\tt{d}\alpha\tt{e}\beta\tt{d}$ and $\tt{d}\lambda\tt{ed}\sigma(\lambda)\tt{e}$, which each appear only once.  To the left of $\tt{d}\alpha\tt{e}$ in $\tt{d}\alpha\tt{e}\beta\tt{d}$ is the letter $\gamma$, which is different from $\tt{d}$, the letter to the left of $\tt{d}\sigma(\lambda)\tt{e}$ in $\tt{d}\lambda\tt{ed}\sigma(\lambda)\tt{e}$.  On the other hand, the length $2$ factor to the right of $\tt{d}\alpha\tt{e}$ in $\tt{d}\alpha\tt{e}\beta\tt{d}$ is $\beta\tt{d}$, while the length $2$ factor to the right of $\tt{d}\sigma(\lambda)\tt{e}$ in $\tt{d}\lambda\tt{ed}\sigma(\lambda)\tt{e}$ does not end in $\tt{d}$.  We conclude that $|x|\leq 4$.  However, since we inserted the new letter $\tt{e}$ somewhere in $\eta(w)[23:-23]$, this certainly forces $|y|\geq 2|x|$, a contradiction.  The proof that $x$ cannot contain $\tt{e}\beta\tt{d}$ is similar.

The only remaining possibilities are $x=\tt{e}$, $x=\alpha\tt{e}$, $x=\tt{e}\beta$, and $x=\alpha\tt{e}\beta$, and these are all eliminated by inspection.

Otherwise, if (b) is true, insert an $\tt{e}$ into the first appearance of $z$ in $W_{2n}$ as follows:
\[
\gamma\beta\alpha\gamma\tt{d}\alpha \beta\tt{d}\gamma \, \tt{e}\, \alpha \tt{d}\beta \gamma\tt{d}\alpha\gamma\beta\alpha.
\]
Otherwise, if (c) is true, insert an $\tt{e}$ into the first appearance of $z$ in $W_{2n}$ as follows:
\[
\gamma\beta\alpha\gamma\tt{d}\alpha \beta\tt{d}\gamma  \alpha \tt{d}\beta \, \tt{e}\, \gamma\tt{d}\alpha\gamma\beta\alpha.
\]
The proofs for (b) and (c) are similar to the proof for (a).
\end{proof}

\section{Conclusion}

In this article, we proved that $\CRTS(4)=\tfrac{3}{2}$, and $\CRTS(5)=\tfrac{4}{3}$, providing the last unknown values of the strong circular repetition threshold:
\[
\CRTS(k)=\begin{cases}
\tfrac{5}{2} &\mbox{ if } k=2;\\
3 &\mbox{ if } k=3; \mbox{ and}\\
\frac{\lceil k/2\rceil+1}{\lceil k/2\rceil} &\mbox{ if } k\geq 4.
\end{cases}
\]
We also gave a proof that $\CRTI(3)=\CRTW(3)=\tfrac{7}{4}$ by adapting the method used to prove $\CRTS(4)=\tfrac{3}{2}$.  While we conjecture that $\CRTI(k)=\CRTW(k)=\RT(k)$ for all $k\geq 4$, techniques different from those presented here will likely be needed.

\providecommand{\bysame}{\leavevmode\hbox to3em{\hrulefill}\thinspace}
\providecommand{\MR}{\relax\ifhmode\unskip\space\fi MR }
\providecommand{\MRhref}[2]{%
  \href{http://www.ams.org/mathscinet-getitem?mr=#1}{#2}
}
\providecommand{\href}[2]{#2}

\titleformat{\section}{\large\bfseries}{\appendixname~\thesection .}{0.5em}{}
\appendix

\section{Exceptional Cases}\label{Appendix}

\subsection{Exceptions for Case I}\label{Appendix1}

There are $12$ possible prefixes of length $39$ of $\eta(w)$ for which we encounter a $\tfrac{4}{3}^+$-power in 
\[
\sigma(p_{39}(\eta(w)))^R \tt{de}p_{39}(\eta(w)),
\]
and $16$ possible suffixes of length $39$ of $\eta(w)$ for which we encounter a $\tfrac{4}{3}^+$-power in 
\[
s_{39}(\eta(w))\tt{ed}\sigma(s_{39}(\eta(w)))^R.
\]

We circumvent this problem by defining $\eta'(w)$ and $\eta''(w)$, which are both obtained from $\eta(w)$ by making a minor adjustment to a short prefix and/or suffix of $\eta(w)$.  Begin by setting $\eta'(w)=\eta(w)$ and $\eta''(w)=\eta(w),$ and then make the changes described below.  Redefine
\[
W_{2n}=\tt{de} \ \eta'(w) \ \tt{ed} \ \sigma(\eta''(w))^R,
\]
and we claim that $(W_{2n})$ is $\tfrac{4}{3}^+$-free.

\begin{enumerate}
\item\label{fixa} If
\begin{align*}
p_{39}(\eta(w))&=\tt{adbac\light{d}abcadcbac\light{e}bcabdacbadcabcebacbdabc},\\ 
p_{39}(\eta(w))&=\tt{adbac\light{d}abcadcbac\light{e}bcabdacbadcabcdbacbdabc},\\
p_{39}(\eta(w))&=\tt{adcab\light{d}acbadbcab\light{e}cbacdabcadbacbecabcdacb}, \mbox{ or} \\
p_{39}(\eta(w))&=\tt{adcab\light{d}acbadbcab\light{e}cbacdabcadbacbdcabcdacb},
\end{align*}
then swap $\eta'(w)[6]=\tt{d}$ and $\eta'(w)[16]=\tt{e}$. 

\item\label{fixb} If 
\begin{align*}
p_{39}(\eta(w))&=\tt{adcba\light{c}dbcabecbadbcdacbdcabcdbaecbdacdbc},\\
p_{39}(\eta(w))&=\tt{adcba\light{c}dbcabecbadbcdacbdcabcdbadcbdacdbc}, \mbox{ or}\\
p_{39}(\eta(w))&=\tt{adcba\light{c}dbcabecbadbcdabdcadbacbdabcdbadca},
\end{align*}
then change $\eta'(w)[6]$ from $\tt{c}$ to $\tt{e}$.

\item\label{fixc} If
\begin{align*}
p_{39}(\eta(w))&=\tt{adbca\light{b}dcbacebcadcbdabcdbacbdcaebcdabdcb},\\
p_{39}(\eta(w))&=\tt{adbca\light{b}dcbacebcadcbdabcdbacbdcadbcdabdcb}, \mbox{ or}\\
p_{39}(\eta(w))&=\tt{adbca\light{b}dcbacebcadcbdacdbadcabcdacbdcadba},
\end{align*}
then change $\eta''(w)[6]$ from $\tt{b}$ to $\tt{e}$.

\item \label{fixd} If
\begin{align*}
p_{39}(\eta(w))&=\tt{adba\light{ce}abcadcbacdbcabecbadbcdacbdcabcdba},
\end{align*}
then change $\eta'(w)[5:6]$ from $\tt{ce}$ to $\tt{ed}$.

\item \label{fixe} If
\begin{align*}
p_{39}(\eta(w))&=\tt{adca\light{be}acbadbcabdcbacebcadcbdabcdbacbdca},
\end{align*}
then change $\eta''(w)[5:6]$ from $\tt{be}$ to $\tt{ed}$.

\item\label{fixf} If 
\begin{align*}
s_{39}(\eta(w))&=\tt{cbdcadbcdabdcbacdbcadcbdabcebacbd\underline{c}abcda},\\
s_{39}(\eta(w))&=\tt{acdabdcbadbcabdacdbadcbdabcebacbd\underline{c}abcda},\\
s_{39}(\eta(w))&=\tt{acdabdcadbcdacbadcabdacdbcaecbacd\underline{a}bcadb}, \mbox{ or}\\
s_{39}(\eta(w))&=\tt{badbcdacbdcabcdbadcbdacdbcaecbacd\underline{a}bcadb},
\end{align*}
then change $\eta'(w)[-6]$ to $\tt{e}$.

\item\label{fixg} If 
\begin{align*}
s_{39}(\eta(w))&=\tt{bcdbadcbdacdbcabdcbadbcdacbecabcd\underline{b}acbda},\\
s_{39}(\eta(w))&=\tt{abdacdbcadcbacdabdcadbcdacbecabcd\underline{b}acbda},\\
s_{39}(\eta(w))&=\tt{cadcbdacdbadcabcdacbdcadbaceabcad\underline{c}bacdb}, \mbox{ or}\\
s_{39}(\eta(w))&=\tt{bcdbadcabdacbadbcdabdcadbaceabcad\underline{c}bacdb},
\end{align*}
then change $\eta''(w)[-6]$ to $\tt{e}$.

\item \label{fixh} If
\begin{align*}
s_{39}(\eta(w))&=\tt{bcadcbacdbcabdacbadcabc\underline{e}bacbdabca\underline{d}bacda},\\
s_{39}(\eta(w))&=\tt{cbadbcabdcbacdabcadbacb\underline{e}cabcdacba\underline{d}cabda},\\
s_{39}(\eta(w))&=\tt{acbdcabcdacbadbcabdcbac\underline{e}abcadbacb\underline{d}abcdb}, \mbox{ or}\\
s_{39}(\eta(w))&=\tt{cabdacbadcabcdbacbdabca\underline{e}cbacdbcab\underline{d}cbadb},\\
\end{align*}
then swap $\eta'(w)[-6]=\tt{d}$ and $\eta'(w)[-16]=\tt{e}$.

\item \label{fixi} If
\begin{align*}
s_{39}(\eta(w))&=\tt{acdbcabdcbadbcdacbecabcdbacbdabca\light{eb}acda},\\
s_{39}(\eta(w))&=\tt{abdcbacdbcadcbdabcebacbdcabcdacba\light{ec}abda},\\
s_{39}(\eta(w))&=\tt{bcdacbadcabdacdbcaecbacdabcadbacb\light{ea}bcdb}, \mbox{ or}\\
s_{39}(\eta(w))&=\tt{badcabcdacbdcadbaceabcadcbacdbcab\light{ec}badb},
\end{align*}
then change both $\eta'(w)[-6:-5]$ and $\eta''(w)[-6:-5]$ to $\tt{de}$.

\end{enumerate}

\noindent
We briefly explain why the proof of Theorem~\ref{Even} Case I (with some slight modifications) still applies to the newly defined $W_{2n}$.  The changes fall into two main types:

\begin{enumerate}[label={\upshape(\Alph*)}]
\item We swap a letter $\tt{d}$ and a letter $\tt{e}$ in $\eta'(w)$, effectively eliminating the $\tfrac{7}{5}$-power that resulted in the given appearance of $\tt{e}$ in $\eta(w)$ by changing the appearance of $\tt{d}$ in the length $4$ prefix to an $\tt{e}$ instead of the appearance of $\tt{d}$ in the length $4$ suffix (or vice versa).  This is what we do in cases~\ref{fixa} and~\ref{fixh}.

\item We change a letter beside a $\tt{d}$ to an $\tt{e}$.  This is done in such a way that we get an appearance of $\tt{ed}$ in $\sigma(p_{7}(\eta(w)))^R \tt{de}p_{7}(\eta(w))$, or an appearance of $\tt{de}$ in $s_{7}(\eta(w))\tt{ed}\sigma(s_{7}(\eta(w)))^R$, respectively.  This is what we do in cases~\ref{fixb},~\ref{fixc},~\ref{fixf}, and ~\ref{fixg}.  In cases~\ref{fixd}, and~\ref{fixe}, we need to change an $\tt{e}$ to a $\tt{d}$ first in order to make this idea work, which is why we change two consecutive letters in $\eta'(w)$ or $\eta''(w)$.  Finally, in case~\ref{fixi}, we need to make this change in both $\eta'(w)$ and $\eta''(w)$.

Here we need to modify the argument in the proof of Theorem~\ref{Even} Case I that $x$ cannot contain either of the factors $\tt{ed}$ or $\tt{de}$.  We consider only $\tt{ed}$; the argument for $\tt{de}$ is analogous.  Suppose that $xyx$ is a factor of $W_{2n}$ with $2|x|>|y|$, and that $x$ contains the factor $\tt{ed}$.  Observe that the original appearance of $\tt{ed}$ appears internally as $\tt{d}\beta\tt{ed}\sigma(\beta)\tt{e},$ where $\beta\in\{\tt{a},\tt{b}\}$.  On the other hand, the factors of length $2$ to the left and right of the new appearance of $\tt{ed}$ contain neither $\tt{d}$ nor $\tt{e}$.   Therefore, $x$ must be contained in $\beta\tt{ed}\sigma(\beta)$.  However, then we certainly have $|y|\geq 8$, contradicting the fact that $2|x|>|y|$.

In case~\ref{fixi}, we introduce two new appearances of the factor $\tt{de}$, but we rule out the possibility that both appearances of $x$ in $xyx$ contain one of the new appearances of $\tt{de}$ by a finite search.  Then the analogous argument to the one given above for $\tt{ed}$ applies once again.
\end{enumerate}

We check that Lemma~\ref{weproperties}\ref{FourThirds} still applies with $\eta'(w)$ and $\eta''(w)$ in place of $\eta(w)$.  While Lemma~\ref{weproperties}\ref{Frequency} does not necessarily hold for $\eta'(w)$ and $\eta''(w)$, the desired consequence that no factor of $\eta'(w)$ of length $15$ is a factor of $\sigma(\eta''(w))^R$ still holds (we only need to check the factors that contain a letter that has been changed here).  This is all that matters in the remainder of the proof of Theorem~\ref{Even} Case I.  So it suffices to check that
\[
\sigma(p_{39}(\eta''(w)))^R \tt{de}p_{39}(\eta'(w)) \ \ \ \ \mbox{(in cases \ref{fixa}-\ref{fixe})},
\]
and
\[
s_{39}(\eta'(w))\tt{ed}\sigma(s_{39}(\eta''(w)))^R \ \ \ \ \mbox{(in cases \ref{fixf}-\ref{fixi})}
\]
are $\tfrac{4}{3}^+$-free for every exceptional prefix and suffix.

\subsection{Exceptions for Case II(a)}\label{Appendix2}

There are $3$ possible prefixes of length $46$ of $\eta(w)$ for which we encounter a $\tfrac{4}{3}^+$-power in 
\[
\tau(p_{46}(\eta(w)))^R \tt{de}p_{46}(\eta(w)),
\]
and $60$ possible suffixes of length $39$ of $\eta(w)$ for which we encounter a $\tfrac{4}{3}^+$-power in 
\[
s_{39}(\eta(w))\tt{ed}\tau(s_{39}(\eta(w)))^R.
\]
Of these $60$ possibilities, there are $30$ that end in $\tt{a}$, and swapping $\tt{a}$ and $\tt{c}$ gives the other $30$.  We consider only the suffixes ending in $\tt{a}$, as those ending in $\tt{c}$ are handled analogously.

Again, we define $\eta'(w)$ and $\eta''(w)$, which are both obtained from $\eta(w)$ by making a minor adjustment to a short prefix and/or suffix of $\eta(w)$.  Begin by setting $\eta'(w)=\eta(w)$ and $\eta''(w)=\eta(w),$ and then make the changes described below.  Redefine
\[
W_{2n}=\tt{de} \ \eta'(w) \ \tt{ed} \ \tau(\eta''(w))^R,
\]
and we claim that $(W_{2n})$ is $\tfrac{4}{3}^+$-free.  The proof is similar to that of Case I and is omitted.

\begin{enumerate}
\item If
\begin{align*}
p_{46}(\eta(w))&=\tt{abdcb\light{ae}bcabdacbadcabceacbdcadbacdabcadcbdacdba}, \mbox{ or }\\
p_{46}(\eta(w))&=\tt{abdcb\light{ae}bcabdacbadcabceacbdcadbacdabcadcbeacdba},
\end{align*}
then change both $\eta'(w)[6:7]$ and $\eta''(w)[6:7]$ from $\tt{ae}$ to $\tt{ed}$.

\item If
\begin{align*}
p_{46}(\eta(w))&=\tt{ab\light{e}cbadbcda\light{b}dcadbaceabcadcbdacdbadcabceacbdcad},
\end{align*}
the change both $\eta'(w)[3]$ and $\eta''(w)[3]$ from $\tt{e}$ to $\tt{d}$, and both $\eta'(w)[12]$ and $\eta''(w)[12]$ from $\tt{b}$ to $\tt{e}$.

\item There are $21$ possibilities for $s_{39}(\eta(w))$ ending in one of the factors $\tt{dbacbda}$ or $\tt{dbcabda}.$ 
In each of these cases, change $\eta'(w)[-6]$ from $\tt{b}$ to $\tt{e}$.

\item If
\begin{align*}
s_{39}(\eta(w))&=\tt{cdbadcabdacbadbcdabdcadbacb\underline{ea}bcdbadcbda},\\
s_{39}(\eta(w))&=\tt{cdbadcabeacbadbcdabdcadbacb\underline{ea}bcdbadcbda}, \mbox{ or}\\
s_{39}(\eta(w))&=\tt{bcdacbadbcabdcbacdabcadbacb\underline{ea}bcdbadcbda},
\end{align*}
then change $\eta'(w)[-12:-11]$ from $\tt{ea}$ to $\tt{de}$.

\item If
\begin{align*}
s_{39}(\eta(w))&=\tt{dcabdacdbcaecbacdabcadbacb\underline{ea}bcdbadcabda},
\end{align*}
then change $\eta'(w)[-13:-12]$ from $\tt{ea}$ to $\tt{de}$.

\item If
\begin{align*}
s_{39}(\eta(w))&=\tt{acbadcabdacdbcaecbacdabcadbacb\underline{ea}bcdbadc},
\end{align*}
then change $\eta'(w)[-9:-8]$ from $\tt{ea}$ to $\tt{de}$.

\item If
\begin{align*}
s_{39}(\eta(w))&=\tt{bcadcbacdbcabdacbadcabc\underline{e}bacbdabca\underline{d}bacda},
\end{align*}
then swap $\eta'(w)[-6]$ and $\eta'(w)[-16]$, and swap $\eta''(w)[-6]$ and $\eta''(w)[-16]$.

\item If
\begin{align*}
s_{39}(\eta(w))&=\tt{dabcadbacbdcabcdacba\underline{e}bcabdcbac\underline{d}bcadcbda},
\end{align*}
then swap $\eta'(w)[-9]$ and $\eta'(w)[-19].$

\item If
\begin{align*}
s_{39}(\eta(w))&=\tt{cbadbcabdcbacdabcadbacb\underline{e}cabcdacba\underline{d}cabda},
\end{align*}
then swap $\eta'(w)[-6]$ and $\eta'(w)[-16].$

\item If
\begin{align*}
s_{39}(\eta(w))&=\tt{adcbdacdbadcabceacbdcadbcdabdcba\underline{eb}cabda},
\end{align*}
then change both $\eta'(w)[-7:-6]$ and $\eta''(w)[-7:-6]$ from $\tt{eb}$ to $\tt{de}$.

%
%
\end{enumerate}

\subsection{Exceptions for Case II(b1)}\label{Appendix3}

There are $5$ possible prefixes of length $39$ of $\eta(\dot{w})$ for which we encounter a $\tfrac{4}{3}^+$-power in 
\[
\rho(p_{39}(\eta(\pi(\ddot{w}))))^R \tt{de}p_{39}(\eta(\dot{w})),
\]
and $4$ possible suffixes of length $39$ of $\eta(\dot{w})$ for which we encounter a $\tfrac{4}{3}^+$-power in 
\[
s_{39}(\eta(\dot{w}))\tt{ed}\rho(s_{39}(\eta(\pi(\ddot{w}))))^R.
\]

We define $\eta'(\dot{w})$ and $\eta''(\pi(\ddot{w}))$, which are obtained from $\eta(\dot{w})$ and $\eta(\pi(\ddot{w}))$, respectively, by making a minor adjustment to a short prefix and/or suffix of $\eta(\dot{w})$ or $\eta(\pi(\ddot{w}))$, respectively.  Begin by setting $\eta'(\dot{w})=\eta(\dot{w})$ and $\eta''(\pi(\ddot{w}))=\eta(\pi(\ddot{w})),$ and then make the changes described below.  Redefine
\[
W_{2n}=\tt{de} \ \eta'(\dot{w}) \ \tt{ed} \ \rho(\eta''(\pi(\ddot{w})))^R,
\]
and we claim that $(W_{2n})$ is $\tfrac{4}{3}^+$-free.  The proof is similar to that of Case I and is omitted.

\begin{enumerate}

\item If
\begin{align*}
p_{39}(\eta(\dot{w}))&=\tt{bdcadbcdac\light{be}cabcdbacbdabcadcbacdbcabdac},\\
p_{39}(\eta(\dot{w}))&=\tt{bdcadbcdac\light{be}cabcdbadcbdacdbcabecbadbcda},\\
p_{39}(\eta(\dot{w}))&=\tt{bdcadbcdac\light{be}cabcdbacbdabcadcbacdbcabeac},\mbox{ or}\\
p_{39}(\eta(\dot{w}))&=\tt{bdcadbcdac\light{be}cabcdbadcbdacdbcabecbadbcda},
\end{align*}
then change $\eta'(\dot{w})[11:12]$ from $\tt{be}$ to $\tt{ed}$.

\item If
\begin{align*}
p_{39}(\eta(\dot{w}))&=\tt{bdcba\light{ce}bcadcbdacdbadcabeacbadbcdabdcadb},
\end{align*}
then
\begin{align*}
p_{39}(\eta(\pi(\ddot{w})))&=\tt{adcbdab\light{ce}bacbdcabcdacbaecabdacdbcadcbac},
\end{align*} 
and we change $\eta'(\dot{w})[6:7]$ from $\tt{ce}$ to $\tt{ed}$ and $\eta''(\pi(\ddot{w}))[8:9]$ from $\tt{ce}$ to $\tt{ed}$.

\item If
\begin{align*}
s_{39}(\eta(\dot{w}))&=\tt{cadbacbdcabcdacbadbcabdcbac\light{eb}cadcbdacdb},\\
s_{39}(\eta(\dot{w}))&=\tt{adcbdabcdbacbdcadbcdabdcbac\light{eb}cadcbdacdb},\mbox{ or}\\
s_{39}(\eta(\dot{w}))&=\tt{adcbdabcebacbdcadbcdabdcbac\light{eb}cadcbdacdb},
\end{align*}
then change $\eta'(w)[-12:-11]$ from $\tt{eb}$ to $\tt{de}$.

\item If
\begin{align*}
s_{39}(\eta(\dot{w}))&=\tt{bdacdbadcbdabcaebacdabdcadbcdacb\light{ec}abcdb},
\end{align*}
then
\begin{align*}
s_{39}(\eta(\pi(\ddot{w})))&=\tt{cabcdacbdcadbaceabcadcbacdbcab\light{ec}badbcda},
\end{align*} 
and we change $\eta'(\dot{w})[-7:-6]$ from $\tt{ec}$ to $\tt{de}$ and $\eta''(\pi(\ddot{w}))[-9:-8]$ from $\tt{ec}$ to $\tt{de}$.

\end{enumerate}

\end{document}